\pdfoutput=1
\documentclass[11pt,letterpaper]{article}

\usepackage[english]{babel}
\usepackage[T1]{fontenc}
\usepackage{stmaryrd}
\usepackage{algorithmic}
\usepackage{amssymb}
\usepackage{marginnote}
\usepackage[ruled,vlined,linesnumbered]{algorithm2e}
\usepackage{booktabs,tabularx,threeparttable}
\usepackage[normalem]{ulem}
\usepackage{tikz}
\usepackage{enumitem}
\usetikzlibrary{fit,backgrounds,positioning,shapes.geometric}
\usepackage{fullpage}
\usepackage{mathrsfs}
\usepackage{amsmath}
\usepackage{bbm}
\usepackage{graphicx}
\usepackage{stackengine}
\usepackage{amsthm}
\usepackage{subcaption}
\usepackage{authblk}
\usepackage{thmtools}
\usepackage{thm-restate}
\usepackage[disable]{todonotes}
\usepackage{fancyhdr,lastpage}
\usepackage{amsfonts}
\usepackage[dvipsnames]{xcolor}
\usepackage[colorlinks=true,allcolors=blue]{hyperref}
\usetikzlibrary{calc, intersections, positioning}

\makeatletter
\renewcommand{\maketitle}{%
\begin{center}
{\large\bfseries \@title\par}
\vspace{0.5em}
{\normalsize \@author\par}
\end{center}
}
\renewenvironment{abstract}{%
\par\normalfont\normalsize
\begin{center}\begin{minipage}{0.9\linewidth}\small\noindent\textup{Abstract. }\ignorespaces}{\end{minipage}\end{center}\par}

\def\@seccntformat#1{\csname the#1\endcsname.\ }

\renewcommand\section{\@startsection{section}{1}{\z@}%
{2.0ex plus .5ex minus .2ex}{1.0ex plus .2ex}%
{\normalfont\large\centering}}

\renewcommand\subsection{\@startsection{subsection}{2}{\z@}%
{1.75ex plus .5ex minus .2ex}{-1em}%
{\normalfont\normalsize\bfseries}}
\makeatother

\newtheorem{thm}{Theorem}[section]

\newtheorem{definition}[thm]{Definition}
\newtheorem{lm}[thm]{Lemma}
\newtheorem{assumption}[thm]{Assumption}

\newtheorem{cor}[thm]{Corollary}
\newtheorem{prop}[thm]{Proposition}

\theoremstyle{remark}
\newtheorem{remark}{Remark}

\theoremstyle{plain}

\newcommand{\set}[1]{\left\{#1\right\}}

\newcommand{\abs}[1]{\left| #1 \right|}

\newcommand{\R}{\mathbb{R}}

\newcommand{\N}{\mathbb{N}}
\newcommand{\C}{\mathbb{C}}

\begin{document}

\title{DISTRIBUTION OF ZEROS OF HOLOMORPHIC FUNCTIONS AND RESONANCES IN LOGARITHMIC REGIONS FOR SCHR\"{O}DINGER OPERATORS ON EUCLIDEAN SPACE}
\author{TRAVIS CUNNINGHAM}
\date{}
\maketitle

\begin{abstract}
    Motivated by scattering theory, this paper proves general results about the distribution of zeros in logarithmic neighborhoods of the real axis for a certain class of functions holomorphic in the closed lower half-plane. We define a new \emph{indicator function} that measures the growth of the holomorphic function along logarithmic curves, and connect this to the distribution of zeros of the function. This is related to classical results on the distribution of zeros in sectors for entire functions of completely regular growth. These results can be applied to the determinant of the scattering matrix of a Schr\"{o}dinger operator on odd-dimensional Euclidean space, yielding bounds on resonance counting functions for logarithmic neighborhoods of the real axis. As a further application, we study a certain family of potentials in one dimension having jump-like singularities. Using our complex-analytic results we show that the singularities can lead to many -- and even infinitely many -- strings of resonances along logarithmic curves. We connect the properties of these strings of resonances, including their location and linear density, both to the parameters describing the singularities of the potential and to the asymptotic behavior of the scattering determinant. 

\end{abstract}

\section{INTRODUCTION}

\subsection{Results}

Resonances are defined as the poles of the meromorphic continuation of the scattering resolvent, and are discrete spectral data replacing eigenvalues in systems that allow energy to escape to infinity. See \cite{10} for a thorough overview of the subject. In this paper, we prove new results on the distribution of zeros of holomorphic functions, and then apply these results to study the high-energy distribution of resonances in regions of the form $\set{\lambda\in\C\colon \operatorname{Im}\lambda\geq -M\log(1+\abs{\lambda}),\operatorname{Re} \lambda\geq 0}$, $M>0$. Throughout the paper we let $\mathcal{R}_V$ denote the set of resonances repeated to multiplicity for the Schr\"{o}dinger operator $-\Delta+V$, $V\in L^\infty_c(\R^n;\C)$. 


A classical result of Vainberg (\cite{26}; see also \cite{16}, \cite[Section 4.6]{10}, \cite{G}) states that for $V \in C_{c}^{\infty}$, the Schr\"{o}dinger operator, $-\Delta + V$, has at most finitely many resonances above the log curve $\set{\operatorname{Im} \lambda = -M\log(1 + \abs{\lambda})}$ for any $M > 0$. In \cite{7}, Christiansen and the author prove the following estimate on the density of resonances in any logarithmic neighborhood of the real axis for Schr\"{o}dinger operators in dimension one. For any $M > 0$, 

\begin{equation*}
\#\set{\lambda_j \in \mathcal{R}_V \colon \operatorname{Im} \lambda_j > -M \log(1 + \abs{\lambda_j}), \pm \operatorname{Re} \lambda_j > 0, \abs{\lambda_j} < r} \leq \frac{1}{\pi} \abs{\text{ch sing supp }V}r(1 + o(1)), \tag{1.1} \label{eq:101}
\end{equation*}
where $\text{ch sing supp } V$ denotes the convex hull of the singular support of $V$. Generalizing prior results of \cite{20}, \cite{27}, \cite{25}, a number of explicit examples given in \cite{7} show
that this bound is sharp in the sense that there exist potentials for which (\ref{eq:101}) is an equality; however, for a given $V$, (\ref{eq:101}) may be a strict inequality for many $M$. In Section \ref{sec:4} of this paper, we will explicitly construct a potential with infinitely many distinct strings of resonances, each with a positive linear density and which lie ever deeper in the complex plane, showing that (\ref{eq:101}) may even be a strict inequality for \emph{all} $M$. 

As is the case with so many results in the theory of scattering resonances \eqref{eq:101} is obtained using the characterization of resonances as the zeros of the scattering determinant, and applying complex analysis. Thus, motivated by the above results, we shall study the distribution of zeros in logarithmic regions for the following class of functions holomorphic in the closed lower half-plane. Recall that a function $f$ holomorphic in some angle $-\pi\leq \alpha\leq \arg z\leq \beta\leq \pi$ is said to be of order $\rho$ if

\begin{equation*}
    \rho=\inf\set{v>0\colon \varlimsup_{r\to\infty} \frac{\log\abs{f(re^{i\theta})}}{r^v}=0,\text{ uniformly in }\theta\in (\alpha,\beta)}.
\end{equation*}

A function of order $\rho$ is said to be of \emph{finite type} provided
\[
    f(z)=O\left(e^{\abs{z}^\rho}\right),\quad \abs{z}\to \infty,\quad \arg z\in (\alpha,\beta). 
\]

\begin{definition}\label{def:1.1}
    We say that $f\in\mathcal{F}$ provided for some integer $\rho\geq 1$,
    \begin{enumerate}[label=\roman*)] 
        \item $f$ is holomorphic and of at most order $\rho$ and finite type in $-\pi\leq \arg z\leq 0$,
        \item $f$ satisfies
        \begin{equation*}
            \int_0^r \frac{f'(s)}{f(s)}ds=O\left(r^{\rho-1}\right),\quad \text{as }r\to\pm\infty. 
        \end{equation*} 
    \end{enumerate}
\end{definition}

In Section \ref{sec:2} we will show that for any $V\in L^\infty_c(\R^n;\C)$ the scattering determinant $\det S_V(-\lambda)$, perhaps multiplied by a rational function, is always in $\mathcal{F}$. We will also show that for any $f\in\mathcal{F}$, the following function is well-defined:

\begin{equation*}
    \ell_f(t):=\varlimsup_{\substack{\abs{z}\to\infty \\ \set{\operatorname{Im} z=-t\log\abs{z},\operatorname{Re}z>0}}}\frac{\log\abs{f(z)}}{\abs{z}^{\rho-1}\log\abs{z}},\quad t\geq 0 \tag{1.2} \label{eq:102}
\end{equation*}

The function $\ell_f$ is motivated by the concept of an \emph{indicator function}, see \cite{17}, \cite{2}, and Section \ref{sec:2}. It is convex (see Lemma \ref{lem:2.1}), and hence it is continuous and has everywhere a right and left derivative that agree at all but countably many points. 

We now state two results connecting the function $\ell_f$ to the distribution of zeros of $f$ in $\set{\operatorname{Re} z>0}$ and in logarithmic neighborhoods of the real axis; our results have corresponding statements for $\set{\operatorname{Re} z<0}$. For $0\leq t<\infty$, we set 

\begin{equation*}
    n(t;r)=\#\set{z\colon f(z)=0, \operatorname{Re}z>0, 0\geq \operatorname{Im} z\geq -t\log(1+\abs{z}), \abs{z}<r}
\end{equation*}
counting all zeros with multiplicity. 

\begin{thm}\label{thm:1.2}
    Let $f\in\mathcal{F}$. Then for any $t\geq 0$, we have
    \begin{equation*}
        \varlimsup_{r\to\infty}\frac{1}{r^\rho}\int_0^t\int_0^r \frac{n(\tau,\nu)}{\nu} d\nu d\tau\leq \frac{\ell_f(t)}{2\pi\rho^2} \tag{1.3} \label{eq:103}
    \end{equation*}
\end{thm}

This bound is optimal in the sense that there exist non-trivial $f\in\mathcal{F}$ for which \eqref{eq:103} is an equality for all $t\geq 0$. Notice, moreover, that the bound depends only upon the behavior of $f$ along a single $\log$ curve. 

The next result shows that with a better understanding of the growth of $f$ along logarithmic curves, we can obtain a stronger statement about the distribution of its zeros. 

\begin{thm}\label{thm:1.3}
    Let $f\in\mathcal{F}$, and assume that for some open interval $I\subset (0,\infty)$, the limit
    
    \begin{equation*}
        \lim_{\substack{\abs{z}\to \infty \\ \set{\operatorname{Im} z=-t\log\abs{z},\operatorname{Re} z>0}}} \frac{\log\abs{f(z)}}{\abs{z}^{\rho-1}\log\abs{z}}=\ell_f(t) \tag{1.4} \label{eq:104}
    \end{equation*}
exists for all $t\in I$. Then 
\begin{equation*}
    \lim_{r\to\infty} \frac{n(t;r)}{r^\rho}=\frac{\ell_f'(t)}{2\pi\rho}
\end{equation*}
for all $t\in I$ for which $\ell_f'(t)$ exists. 
\end{thm}

The above complex analytic results are motivated by the theory of entire functions of completely regular growth \cite[Ch.'s II and III]{17} and its generalization to so-called \emph{curves of regular rotation} in \cite{1}, \cite{2}, which relate the growth of holomorphic functions along certain curves to the distribution of its zeros in regions bounded by these curves, though we remark that our growth condition \eqref{eq:104} is stronger than the ones assumed there. See the next subsection and Section \ref{sec:2} for more discussion. 

We now turn to applications of these results to the study of resonances in odd-dimensional Euclidean potential scattering. For any $V\in L_c^\infty(\R^n;\C)$, $n\geq 1$ odd, we will let $\ell_V$ be the function defined through \eqref{eq:102} with $f$ equal to $\det S_V(-\lambda)$, perhaps multiplied by a rational function. With at most finitely many exceptions, the zeros of this $f$ agree with resonances. When $V\in C_c^\infty$, it can be shown that $\ell_V\equiv 0$. On the other hand, Zworski \cite{28} and Stefanov \cite{24} show that for certain radial potentials with a jump singularity, there exists a non-zero $r^n$-density of resonances in a logarithmic region, and Theorem \ref{thm:1.2} then shows that $\ell_V\not\equiv 0$ for these potentials. In this sense, $\ell_V$ captures the singular behavior of the scattering system. See also Theorems \ref{thm:1.6}, \ref{thm:1.7}, and \ref{thm:1.8} below. In Sections \ref{sec:3} and \ref{sec:4}, the function $\ell_V$ is computed explicitly for certain families of singular potentials. 

We may now state a general bound for resonances in logarithmic regions. Recall that since $\ell_V$ is convex, its right-hand derivative $\ell_V'(t+)$ exists for all $t\geq 0$. 

\begin{thm}\label{thm:1.4}
    For any $V\in L_c^\infty(\R^n;\C)$ and any $t\geq 0$,

    \begin{equation*}
        \varlimsup_{r\to\infty} \frac{\# \set{\lambda_j\in\mathcal{R}_V\colon \operatorname{Im} \lambda_j>-t\log(1+\abs{\lambda_j}), \operatorname{Re} \lambda_j>0, \abs{\lambda_j}<r}}{r^n}\leq \frac{e}{\pi n}\ell_V'(2t+). 
    \end{equation*}
\end{thm}

This theorem will follow as a corollary to Theorem \ref{thm:1.2} provided in Section \ref{sec:2}. Using the methods of Section \ref{sec:3}, it is easy to construct examples of potentials in dimension one for which Theorem \ref{thm:1.4} gives a sharper bound than \eqref{eq:101}. In addition, it gives new quantitative estimates on resonance density in logarithmic regions for \emph{any} odd dimension.

Next we turn to a certain class of potentials in dimension one having many jump-like singularities, and for which the relationship between the singular behavior of $V$, the function $\ell_V$, and the distribution of resonances can be described in much more detail. 

\begin{definition}\label{def:1.5}
    We say $V \in \mathcal V$ provided that for some $N > 0$, there exist $x_0 < x_1 < \cdots < x_{N}$ and $\set{\nu_k}_{k=0}^{N} \subset \N_0$ such that $\operatorname{ch supp} V = [x_0, x_N]$ and the following hold:
    \begin{enumerate}[label=\roman*)] 
        \item $\nu_k$ is the smallest integer such that $V^{(\nu_k)}(x_k-) \neq V^{(\nu_k)}(x_{k}+), k = 0, 1, \ldots, N$. 
        
        \item $V \in C^{M_k}\left([x_{k-1}, x_k] \right)$ where $M_k = \max \set{\nu_{k-1}, \nu_k} + 1$, $k = 1, 2, \ldots, N$.

        \item Let $C_{j, k} := (-1)^{\nu_j + 1}\left(V^{(\nu_j)}(x_j+) - V^{(\nu_j)}(x_j-) \right) \left(V^{(\nu_k)}(x_k+) - V^{(\nu_k)}(x_k -) \right)$. If, for some $T \geq 1$, the lengths of some subintervals $\set{[x_{j_t}, x_{k_t}]}_{t=0}^{T}$ are all equal, and if also the corresponding $\set{\nu_{j_t} + \nu_{k_t}}_{t-0}^{T}$ are all equal, then we require $\sum_{t=0}^{T} C_{j_t, k_t} \neq 0$. 
    \end{enumerate}
\end{definition}

We remark that condition (iii) is not always required for the results below to hold; it is assumed only to ensure that we avoid a certain degenerate situation which could perhaps be dealt with by some other means (see Remark \ref{remark:1} after (\ref{eq:407}), and section \ref{sec:4.3}, for details).

In \cite{7}, Christiansen and the author studied the resonances for similar potentials having either two or three such singularities, and satisfying certain additional simplifying assumptions. It is shown there that such potentials can produce either one or two strings of resonances along logarithmic curves, generalizing prior results of \cite{20}, \cite{27}, and \cite{25}. The following theorems show that potentials in $\mathcal V$ produce multiple strings of resonances along logarithmic curves, and provide a detailed connection between these strings and the function $\ell_V$. 

\begin{thm}\label{thm:1.6}
    Let $V \in \mathcal V$. If $t_0$ is a discontinuity of $\ell_V'$, then there exists a sequence of resonances $\set{\lambda_{m}}_{m=1}^{\infty}$ with linear density 
    \begin{equation*}
        \frac{1}{2\pi }\left(\ell_V'(t_0 +)- \ell_V'(t_0-) \right)
    \end{equation*}
    satisfying 
    \begin{equation*}
        \lim_{m \to \infty} \frac{- \operatorname{Im} \lambda_m}{\log \operatorname{Re} \lambda_m} = t_0.
    \end{equation*}
\end{thm}

In other words, a jump discontinuity of $\ell_V'$ implies the existence of a string of resonances. The following converse shows that these strings constitute all of the resonances of $V$. 

\begin{thm}\label{thm:1.7}
    Let $V \in \mathcal V$ and let $\set{\lambda_m}_{m=1}^{\infty}$ be a sequence of resonances in $\set{\operatorname{Re \lambda \geq 0}}$. Then there exists $t_0$ with $\ell_V'(t_0 + ) \neq \ell_V'(t_0-)$ such that for some subsequence $\set{\lambda_{m_k}}_{k=1}^{\infty}$ we have 
    \begin{equation*}
        \lim_{k \to \infty} \frac{-\operatorname{Im} \lambda_{m_k}}{\log \operatorname{Re} \lambda_{m_k}} = t_0. 
    \end{equation*}
\end{thm}

Together with Theorem \ref{thm:1.6}, this gives a new characterization of the strings of resonances in terms of discontinuities of $\ell_V'$, a function defined naturally in terms of the scattering determinant. Closer analysis provides connections between the parameters defining the singularities of $V$ and the discontinuities of $\ell_V'$. In particular, we show that each discontinuity $t_0$, and each of the jumps of $\ell_V'$ across $t_0$, have explicit expressions in terms of these parameters. The precise statements are contained in the following: 

\begin{thm}\label{thm:1.8}
    Let $V \in \mathcal V$, and let $\set{x_k}_{k=0}^N, \set{\nu_k}_{k=0}^N$ be as in the definition of $\mathcal V$. Then there are at most $N(N+1)/2$ discontinuities of $\ell_V'$, hence at most $N(N+1)/2$ strings of resonances. Moreover, the first discontinuity occurs at 
    \begin{equation*}
        t_0 := \min_{0 \leq m < l \leq N} \frac{4 + \nu_l + \nu_m}{2(x_l - x_m)}
    \end{equation*}
    and 
    \begin{equation*}
        \ell_V'(t_0+) - \ell_V'(t_0-) = \max \set{2(x_l - x_m) \colon \frac{4 + \nu_l + \nu_m}{2(x_l - x_m)} = t_0}.
    \end{equation*}
    If $\ell_V'$ has a discontinuity at $t > t_0$, then for some $l_0 > m_0$, $l_1 > m_1$, $t$ has the form
    \begin{equation*}
        t = \frac{\nu_{l_1} + \nu_{m_1} - \nu_{l_0} - \nu_{m_0}}{2[(x_{l_1} - x_{m_1}) - (x_{l_0} - x_{m_0})]}
    \end{equation*}
    and 
    \begin{equation*}
        \ell_V'(t +) - \ell_V'(t-) = 2[(x_{l_1} - x_{m_1}) - (x_{l_0} - x_{m_0})].
    \end{equation*}
\end{thm}

The proof, given in Section \ref{sec:4.3}, will actually give an algorithm for computing every discontinuity of $\ell_V'$. This result gives a more refined picture of the correspondence between the singularities of $V$, the discontinuities of $\ell_V'$, and the location and density of the strings of resonances. 

There is a natural dynamical interpretation of the form of these discontinuities (hence locations of strings of resonances) with connections to what has been observed in other settings. We discuss this more in the next section, but it is interesting to note here that $t_0$ corresponding to the string closest to the real axis has a different form than the ones deeper in the plane. This agrees with observations of \cite{9}, \cite{3}
 in the parallel context of many $\delta$-function potentials on the line. There, the authors determine the location of this so-called \emph{dominant string} of resonances by characterizing it as the minimal slope of a certain Newton polygon defined directly in terms of the parameters, see \cite[Definition 3]{3}. Here, we characterize the location of the dominant string as the first discontinuity of $\ell_V'$, providing a direct connection to the asymptotic behavior of the scattering determinant. 

 Finally, as mentioned briefly above, for our final result we will explicitly construct a potential producing infinitely many strings of resonances along distinct logarithmic curves, each with a positive linear density (Theorem \ref{thm:5.3}).

\subsection{Relation to existing work}

Section \ref{sec:2} provides the proofs of Theorems \ref{thm:1.2}, \ref{thm:1.3}, and \ref{thm:1.4}, and parallels the theory of functions of completely regular growth, \cite[Ch.'s II and III]{17}, \cite{1}, and \cite{2}. In this context, our Proposition \ref{prop:2.4} is a direct analogue of \cite[(3.04)]{17} and \cite[Theorem 3]{2}, $\ell_f$ plays the role of an \emph{indicator function} (see \cite[(1.64)]{17} and \cite[Section 1, Definition 1]{2}), \eqref{eq:104} provides us with sufficient regularity of growth, and Theorem \ref{thm:1.3} completes the analogy with \cite[Ch. III, Section 3, Theorem 3]{17} and \cite[Theorem 4]{2}. 
Since logarithmic curves do not fit the hypotheses of curves of regular rotation, we must build the appropriate tools from scratch; some aspects of our version are technically simplified while others are more complicated than their counterparts in \cite{17}, \cite{1}, \cite{2}. This general framework is built up with the standard odd dimensional potential scattering in mind since this is the setting of our applications, but we remark that it is likely possible to adapt our methods to accommodate more general black box operators in any dimension. See the remarks at the end of Section \ref{sec:2}, and see \cite{23}, \cite[Chapter 4]{10} for information on resonances in black box scattering. Finally, we mention here that the theory of functions of completely regular growth has been used previously in the study of resonances, most notably by Christiansen in \cite{6} to describe the distribution of resonances in sectors. 

Our results on the class $\mathcal V$ sharpen and generalize recent results of \cite{7} on resonances in dimension one. By further refining the integral representations for the entries of the scattering matrix used in \cite{11}, \cite{22}, and \cite{7}, we obtain a precise description of the asymptotic behavior of the scattering determinant along logarithmic curves. This allows us to apply Theorem \ref{thm:1.3} and prove Theorems \ref{thm:1.6}, \ref{thm:1.7}, and \ref{thm:1.8}. Whereas \cite{7} uses Hardy's method to give some explicit asymptotic expressions for the resonances in the case of two or three singularities, we focus here on the structure of more general problems and on better understanding the connections between the nature of the singularities of $V$ and the location of the strings of resonances. We remark that for a generic choice of parameters, it is possible to apply the simplest version of Hardy's method, \cite[Lemma A.1]{7} to (\ref{eq:413}) to explicitly compute the asymptotic location of the resonances. The point with our alternative method is the ease with which we deal with the more complicated scenarios, and more importantly the connection to the asymptotic behavior of the scattering determinant. 

Potentials in $\mathcal V$ have similarities with one dimensional conormal distributions, for which a detailed study of propagation of singularities and applications to resonance-free regions may be found in \cite{13} and \cite{14} respectively. In this context, the structure of the discontinuities of $\ell_V'$ described in Theorem \ref{thm:1.8}, or equivalently the $t$ for which there is a string of resonances asymptotic to $\set{\operatorname{Im} \lambda = -t\log(1 + \abs{\lambda}), \operatorname{Re} \lambda > 0}$ is natural from a dynamical perspective and is connected to the notion of \emph{diffraction of singularities}. The denominators are differences of the lengths traversed by a wave that travels from one singularity to another and back, and the numerators involve the $\nu_k$, which quantify the amount of the wave that passes through the corresponding singularity. This structure is similar to what has been observed in other settings, including semiclassical resonances for many $\delta$-function potentials on the line, \cite{9}, \cite{3} (where the structure is very similar to ours - see below), and manifolds with conic singularities, \cite{15} (where the numerator depends upon the dimension of the manifold, and the denominator is the length of the geodesic connecting two cone points). 

We now give a more detailed comparison of our results on the class $\mathcal V$ to the results of \cite{9}, \cite{3}. As mentioned above, these papers study semiclassical resonances for operators of the form 
\begin{equation*}
    -h^2 \frac{d^2}{dx^2} + \sum_{k=0}^{N} V_k \delta\left(x - x_k \right)
\end{equation*}
where $x_{0} < x_{1} < \cdots < x_{N}$ and each $V_k = C_kh^{1 + \beta_{k}}$ for some $C_k \in \R/\set{0}$ and $\beta_{k} > 0$. It is shown in \cite{9} that resonances occur on finitely many lines of the form 
\begin{equation*}
    \operatorname{Im}z \sim - \gamma h\log{\frac{1}{h}} \quad \gamma > 0,
\end{equation*}
and the form of these $\gamma$ are determined in terms of the $x_{k}, \beta_{k}$, see \cite[Theorem 7]{9}. An upper bound of $2^{N} - 1$ on the possible number of distinct strings of resonances is given in \cite{9}; using some additional assumptions, \cite{3} refined the description of the possible $\gamma$, and used this refinement to sharpen the bound on the number of strings to $N$. 

Comparing to our results on potentials in $\mathcal V$, we note that our description for the location of the strings of resonances in Theorem \ref{thm:1.8} is more precise than the analogue in \cite{9}, but differs in form from the refinement in \cite{3}. This difference is due to the use in \cite{3} of a simplifying assumption on the choice of parameters (see \cite[Assumption 1]{3}) to avoid certain complicated situations. We make no such assumption, and our method deals with the analogous situation in a systematic way. The tradeoff, however, is that without an analogue of \cite[Assumption 1]{3}, we could not adapt the method of \cite{3} to obtain the more precise description of the locations of the strings of resonances. This is the reason our bound of $N(N+1)/2$ for the number of strings improves over the $2^{N}-1$ bound of \cite{9}, but does not match the bound of $N$ given by \cite{3}. However, we remark that we have verified directly the bound of $N$ for any relationship of the parameters up to $N = 6$, and we conjecture that $N$ is in fact the optimal bound in both settings, without restriction on the parameters. 

Both \cite{9} and \cite{3} prove their results by characterizing the locations of the strings of resonances in terms of the slopes of a Newton polygon defined directly in terms of the parameters $x_{k}, \beta_{k}$. Our interpretation of the strings in terms of $\ell_V$ has the benefit of capturing the location and density of the strings of resonances simultaneously, and provides a direct connection to the scattering determinant. In addition, we emphasize that $\ell_V$ is a natural object defined and relevant to the distribution of resonances for \emph{any} $V \in L_{c}^{\infty}$ in any odd dimension, not just for the class $\mathcal{V}$. Finally, we mention that the analogue of Theorem \ref{thm:1.7} is not known in the context of \cite{9}, \cite{3} (see \cite[Conjecture 1]{3}). Moreover, our results give some evidence of a positive answer to \cite[Conjecture 2]{3}: In the setting of \cite{3}, when $\operatorname{Re}z \to \infty$ with $h$ fixed, all singularities are of equal strength since they are just $\delta$-function singularities at the various points. The analogue in our setting is to take all $\nu_k$ equal and Theorem \ref{thm:1.8} shows that in this case there is indeed exactly one string of resonances, corresponding to the contribution of the endpoints $x_{0}, x_{N}$.

Many recent papers have explored the connection between certain singular behaviors for the scattering system and the production of one or more strings of resonances along logarithmic curves; in addition to the results of \cite{9}, \cite{3}, \cite{15}, already discussed above, we also mention \cite{4} (finding multiple strings of resonances for scattering by analytic obstacles, one having a corner), \cite{8} (finding a logarithmic string of resonances for a non-trapping surface of revolution with a cone point and a funnel), and see also related results in \cite{12}. We refer to \cite[Section 1.2]{14} for a more detailed overview. Comparatively, our results are very precise, exhibiting explicit correspondences between the nature of the singular behavior of the system and the location of the strings of resonances, as well as providing a direct connection to the scattering determinant. Moreover, to our knowledge, the construction of Section \ref{sec:4} provides the first example of a single system producing an infinite number of strings of resonances along distinct logarithmic curves. 

\noindent \textbf{Acknowledgments.} The results in this paper originate in the many discussions with Tanya Christiansen during and since our collaboration in \cite{7}. It is a pleasure to thank her for her generous advice and for several comments which have improved the exposition. We also thank Dan Cunningham for logistical help, and Ben Jeffers, Tian An Wong, and the Prison Mathematics Project for handling the typing of the original manuscript of this paper.

\section{DISTRIBUTION OF ZEROS OF HOLOMORPHIC FUNCTIONS}\label{sec:2}

In this section, we follow ideas of \cite[Ch.'s II and III]{17}, \cite{1}, \cite{2} to study the distribution of zeros in logarithmic neighborhoods of the real axis for the functions $f\in\mathcal{F}$ from Definition \ref{def:1.1} and prove Theorems \ref{thm:1.2} and \ref{thm:1.3}. We will then show that for any $V\in L_c^\infty(\R^n;\C)$, $n\geq 1$ odd, the scattering determinant, perhaps multiplied by a rational function, satisfies these hypotheses. Using Corollary \ref{cor:2.5}, the characterization of resonances as the zeros of $\det S_V(-\lambda)$ in $\operatorname{Im} \lambda\leq 0$ (see \cite[Theorems 2.14 and 3.45]{10}) then completes the proof of Theorem \ref{thm:1.4} as well. 

Recall that Definition \ref{def:1.1} says that $f\in\mathcal{F}$ provided for some integer $\rho\geq 1$, $f$ is holomorphic and of at most order $\rho$ and finite type in $-\pi\leq \arg z\leq 0$, and satisfies

\begin{equation*}
    \int_0^r \frac{f'(s)}{f(s)} ds =O\left(r^{\rho-1}\right), \quad\text{as }r\to \pm\infty \tag{2.1} \label{eq:201}
\end{equation*}

As mentioned above, our results are motivated by the theory of entire functions of completely regular growth along curves of regular rotation (see \cite{17}, \cite{1}, \cite{2}), and hypotheses such as those in Definition \ref{def:1.1} are natural in this context (in fact, ours are somewhat more general in that we allow certain functions holomorphic only in a half-plane).

Notice that (\ref{eq:201}) implies the bound 
\begin{equation*}
    \abs{\log \abs{f(x)}} \leq C\left(1 + \abs{x} \right)^{\rho-1}, \quad x \in \R \tag{2.2} \label{eq:202}
\end{equation*}
A simple application of the Phragm\'{e}n-Lindel\"{o}f principle in some sector $\arg z \in (-\varepsilon, 0)$ shows that any $f \in \mathcal F$ satisfies 
\begin{equation*}
    \log \abs{f(z)} \leq C_1 \abs{z}^{\rho -1}\abs{\operatorname{Im} z} + C_2 \abs{z}^{\rho-1}, \quad \arg z \in (-\varepsilon, 0), \abs{z} > C_{0},
\end{equation*}
and thus the function $\ell_f$ from \eqref{eq:102} is well-defined.

By analogy with \cite[(1.64)]{17} and \cite[Definition 1]{2}, we call $\ell_f$ the \emph{logarithmic indicator function}. As explained in the introduction, $\ell_f$ can be used to describe the distribution of zeros of $f$ in logarithmic neighborhoods of the real axis. Before giving results, we first note that the following gives a smooth parametrization of the regions $\set{0 \geq \operatorname{Im}z \geq -t_0 \log(1 + \abs{z}), \operatorname{Re} z > 0, \abs{z} > r_0}$ for any fixed $t_0, r_0 > 0$: 
\begin{equation*}
    z(r, t) := re^{i\theta_t(r)}, \quad \theta_t(r) := \sin^{-1}\left(\frac{-t\log(1 + r)}{r} \right), \quad t \in [0, t_0], r > r_0.
\end{equation*}

We now prove the following property of $\ell_f$:

\begin{lm}\label{lem:2.1}
    Let $f \in \mathcal F$. Then $\ell_f(0) = 0$ and $\ell_f$ is convex on $(0, \infty)$; in particular, it is continuous and has everywhere a right and left derivative which agree at all but countably many points. 
\end{lm}
\begin{proof}
    That $\ell_f(0) = 0$ follows from (\ref{eq:202}). Fix $s, t$ with $0 \leq s < t < \infty$. Then convexity on $[s, t]$ is equivalent to showing that 
    \begin{equation*}
        \ell_f(u) \leq Au - B, \quad u \in [s, t],
    \end{equation*}
    where 
    \begin{equation*}
        A = \frac{\ell_f(t) - \ell_f(s)}{t-s}, \quad B = \frac{s\ell_f(t) - t\ell_f(s)}{t-s}.
    \end{equation*}

    Let $\delta > 0$ be given and define 
    \begin{equation*}
        \varphi_{\delta}(z) := \exp \left((B - \delta)z^{\rho-1} \log z - i \frac{A}{\rho}z^\rho \right)f(z)
    \end{equation*}
    which is holomorphic in $\set{\arg z \in [-\varepsilon, 0], \abs{z} > 0}$, for small $\varepsilon > 0$. Notice that 
    \begin{equation*}
        \varlimsup_{r \to \infty} \frac{\log \abs{\varphi_{\delta} \left(re^{i \theta_{u}(r)} \right)}}{r^{\rho-1} \log{r}} = \ell_f(u) - Au + B - \delta. \tag{2.3} \label{eq:203}
    \end{equation*}
    In particular, $\varphi_\delta\left(re^{i\theta_{s}(r)} \right) \to 0$, $\varphi_\delta\left(re^{i\theta_{t}(r)} \right) \to 0$ as $r \to \infty$, and by the Phragm\'{e}n-Lindel\"{o}f principle, we conclude that $\varphi_\delta(z)$ is uniformly bounded in the region 
    \[
    \set{-s \log(1 + \abs{z}) \geq \operatorname{Im}z \geq -t\log(1 + \abs{z}), \operatorname{Re}z > C_{\delta}}.
    \]
     From (\ref{eq:203}), then, we see that $\ell_f(u) \leq Au - B + \delta$ for all $u \in [s, t]$, and since $\delta$ is arbitrary, the lemma is proved. 
\end{proof}

Next we prove two technical lemmas that we use in the proofs of the main results of this section. Note that both describe certain regular growth properties for a function $f \in \mathcal F$. 

\begin{lm}\label{lem:2.2}
Let $f \in \mathcal F$. Then for any $\varepsilon > 0$, there exists $C_{\varepsilon}$ and a set $E_{\varepsilon} \in (0, \infty)$ of finite measure such that 
\begin{equation*}
    \abs{\log{\abs{f(z)}}} \leq C_{\varepsilon}\abs{z}^{\rho + \varepsilon}
\end{equation*}
as $\abs{z} \to \infty, \operatorname{Im}z \leq 0, \abs{z} \notin E_{\varepsilon}$. 
\end{lm}

\begin{proof}
    We first note that using for instance \cite[(2-2)]{6}, it follows that for any $f \in \mathcal F$, the number of zeros in $\operatorname{Im}z \leq 0$ with $\abs{z} < r$ is $\leq Cr^{\rho}$ for some $C$. We define 

\begin{equation*}\tag{2.4}\label{eq:204}
    \begin{aligned}
        F(z) := &\; f(z)\frac{\overline{P(\overline{z})}}{P(z)}, \\
        P(z) = \prod_{\set{z_j \colon f(z_j)=0}} E_p\left( \frac{z}{z_j} \right), &\quad E_p(z) = (1-z)\exp\left(z + \frac{z^2}{2} + \cdots + \frac{z^\rho}{\rho} \right)
    \end{aligned}
\end{equation*}

       Standard estimates on the Weirstrass product $P(z)$ (see, e.g., \cite[Section D.2]{10}) show that for any $\varepsilon > 0$, 
       \begin{equation*}\tag{2.5}\label{eq:205}
       \begin{aligned}
           \abs{\log{\abs{P(z)}}} \leq C_{\varepsilon} \abs{z}^{\rho + \varepsilon}, & \quad \operatorname{Im}z \leq 0, z \notin \bigcup_{\set{z_j \colon f(z_j) = 0}} D\left( z_j, (1 + \abs{z_j})^{- \rho - \varepsilon} \right),\\ \
           \abs{\log{\abs{\overline{P(\overline{z})}}}} \leq C_{\varepsilon} \abs{z}^{\rho + \varepsilon}, & \quad \operatorname{Im}z \leq 0
       \end{aligned}
      \end{equation*}  
    and thus 
    
    \begin{equation*}
        \log{\abs{F(z)}} \leq C_{\varepsilon} \abs{z}^{\rho + \varepsilon}, \quad \operatorname{Im} z \leq 0, z \notin  \bigcup_{\set{z_j \colon f(z_j) = 0}} D\left( z_j, (1 + \abs{z_j})^{- \rho - \varepsilon} \right).
    \end{equation*}
    By the fact that $\abs{\log{\abs{F(t)}}} = \abs{\log{\abs{f(t)}}} \leq C(1 + \abs{t})^{\rho-1}$ on $\R$, and application of the maximum principle, we see that actually 
    \begin{equation*}
        \log{\abs{F(z)}} \leq C_{\varepsilon}\abs{z}^{\rho + \epsilon}
    \end{equation*}
    in all of $\operatorname{Im}z \leq 0$. In particular, $F$ is of finite order in $\operatorname{Im}z \leq 0$. Using the fact that $F$ has no zeros in $\operatorname{Im}z \leq 0$, the representation of a function holomorphic in a half-plane \cite[Appendix VIII, Section 2, (3)]{17} shows that in fact, 
    \begin{equation*}
        \abs{\log{\abs{F(z)}}} \leq C_{\varepsilon} \abs{z}^{\rho + \varepsilon}, \quad \operatorname{Im} z \leq 0, \tag{2.6} \label{eq:206}
    \end{equation*}
    and thus from (\ref{eq:204}), (\ref{eq:205}), and (\ref{eq:206}), 
    \begin{equation*}
        \abs{\log{\abs{f(z)}}} \leq C_{\varepsilon} \abs{z}^{\rho + \varepsilon}, \quad \operatorname{Im}z \leq 0, z \notin \bigcup_{\set{z_j \colon f(z_j) = 0}} D\left(z_j, (1 + \abs{z_j})^{- \rho - \varepsilon}\right).
    \end{equation*}
    Using that the number of zeros of modulus $< r$ is of order $r^\rho$, it is easy to verify that the set 
    \begin{equation*}
        E_{\varepsilon} := \set{r > 0 \colon r = \abs{z} \text{ for some } \operatorname{Im} z \leq 0, z \in \bigcup_{\set{z_j \colon f(z_j) = 0}} D\left(z_j, (1 + \abs{z_j})^{- \rho - \varepsilon}\right)}
    \end{equation*}
    has finite measure, and the lemma follows. 
\end{proof}

The next lemma gives control over integrals involving $\log{\abs{f}}$:

\begin{lm}\label{lem:2.3}
    Let $f \in \mathcal F$. Then for any $r_0 > 0, \varepsilon >0$, there exists $C > 0$ such that 
    \begin{equation*}
        \int_{r_0}^{r} \frac{\abs{\log{\abs{f(ue^{i \theta})}}}}{u} du \leq Cr^{\rho + \epsilon} \tag{2.7} \label{eq:207}
    \end{equation*}
    for $r \geq r_0$, and uniformly in $\theta \in [-\pi, 0]$. 
\end{lm}

\begin{proof}
    The proof of Lemma \ref{lem:2.2} (see (\ref{eq:204}), (\ref{eq:205}), (\ref{eq:206})) shows that for $P(z)$ as defined in \eqref{eq:204}, we have 
    \begin{equation*}
        \abs{\log{\abs{f(z)}}} = \abs{\log{\abs{P(z)}}} + O\left( \abs{z}^{\rho + \varepsilon} \right)
    \end{equation*}
    in $\operatorname{Im}z \leq 0$, so it suffices to prove (\ref{eq:207}) with $f$ replaced by $P$. 

    For any $\tau > 0$ and $r = \abs{z}$,
    \begin{equation*}
        \log{\abs{P(z)}} = \operatorname{Re} \left(z \sum_{\abs{z_j} < \tau r} z_j^{-1} + \cdots z^{\rho} \sum_{\abs{z_j} < \tau r} z_j^{-\rho} \right) + \sum_{\abs{z_j} < \tau r} \log{\abs{1 - \frac{z}{z_j}}} + \sum_{\abs{z_j} \geq \tau r} \log{\abs{E_{\rho}\left( \frac{z}{z_j}\right)}}. \tag{2.8} \label{eq:208}
    \end{equation*}
    If we fix $\tau > 2$, then \cite[Ch. I, Section 17, Lemma 8]{17} shows that 
    \begin{equation*}
        \abs{\sum_{\abs{z_j} \geq \tau r} \log{\abs{E_{\rho}\left( \frac{z}{z_j}\right)}}} \leq C_{\tau}r^{\rho}. \tag{2.9} \label{eq:209}
    \end{equation*}
    Since the number of zeros $z_j$ of modulus $< r$ is of order $r^\rho$ (see the first part of the proof of Lemma \ref{lem:2.2}), by \cite[Ch. I, Section 4, Lemma 1]{17} we have $\sum_{\abs{z_j} < \tau r} \abs{z_j}^{-\rho - \varepsilon} \leq C_{\varepsilon}$ for any $\varepsilon$ and all $r$. Thus, using H\"{o}lder's inequality, for each $k \leq \rho$,
    \begin{equation*}
        \abs{\sum_{\abs{z_j} < \tau r} z_j^{-k}} \leq Cr^{\rho - k + \varepsilon}.
    \end{equation*}
    It follows that 
    \begin{equation*}
        \abs{\operatorname{Re}\left(z \sum_{\abs{z_j} < \tau r} z_j^{-1} + \cdots + z^{\rho} \sum_{\abs{z_j} < \tau r} z_j^{-\rho} \right)} \leq Cr^{\rho + \epsilon}. \tag{2.10} \label{eq:210}
    \end{equation*}
    Using (\ref{eq:208}), (\ref{eq:209}), and (\ref{eq:210}), we find 
    \begin{align*}
        \log{\abs{P(z)}} \geq &\; -Cr^{\rho + \varepsilon} + \sum_{\abs{z_j} < \tau r} \log{\abs{1 - \frac{z}{z_j}}}\\
        \geq &\; -Cr^{\rho + \varepsilon} + \sum_{\abs{z_j} < \tau r}\log{\abs{1 - \frac{r}{\abs{z_j}}}}.
    \end{align*}
    Since 
    \begin{equation*}
        \int_{r_0}^{r} \log{\abs{1 - \frac{\nu}{\abs{z_j}}}} \frac{d\nu}{\nu} = \int_{r_0/\abs{z_j}}^{r/\abs{z_j}} \log \abs{1 - u}\frac{du}{u} \geq \int_0^2 \log{\abs{1-u}}\frac{du}{u} = -\frac{7}{24}\pi^2,
    \end{equation*}
    integrating the above equality, we obtain uniformly for $\theta \in [-\pi, 0]$, 
    \begin{equation*}
        \int_{r_0}^r \frac{\log\abs{P\left(ue^{i\theta} \right)}}{u} du \geq -Cr^{\rho + \varepsilon}, \quad r > r_0.
    \end{equation*}
    Since also $\log{\abs{P(z)}} \leq C\abs{z}^{\rho + \varepsilon}$ in $\operatorname{Im}z \leq 0$, we have uniformly for $\theta \in [-\pi, 0]$, 
    \begin{equation*}
        \int_{r_0}^{r} \frac{\log{\abs{P\left(ue^{i\theta} \right)}}}{u} du = O(r^{\rho + \varepsilon}). \tag{2.11} \label{eq:211}
    \end{equation*}
    Now let $G_{\pm}^{\theta}$ be the set of $u > r_0$ on which $\pm \log{\abs{P\left(ue^{i\theta} \right)}} > 0$. Then 
    \begin{equation*}
        \int_{G_{+}^{\theta} \cap (r_0, r)} \frac{\log{\abs{P \left(ue^{i \theta} \right)}}}{u} du \leq C \int_{r_0}^r u^{\rho - 1 + \varepsilon} du \leq C'r^{\rho + \epsilon}, \tag{2.12} \label{eq:212}
    \end{equation*}
    uniformly in $\theta \in [-\pi, 0]$. Whence also 
    \begin{equation*}
        \int_{G_{-}^{\theta} \cap (r_0, r)} \frac{-\log{\abs{P \left(ue^{i \theta} \right)}}}{u} du = -\int_{r_0}^r \frac{\log{\abs{P \left(ue^{i \theta} \right)}}}{u} du + \int_{G_{+}^\theta \cap (r_0, r)} \frac{\log{\abs{P \left(ue^{i \theta} \right)}}}{u} du \leq Cr^{\rho + \varepsilon} \tag{2.13} \label{eq:213}
    \end{equation*}
    uniformly in $\theta \in [-\pi, 0]$, by (\ref{eq:211}), (\ref{eq:212}). Together (\ref{eq:212}), (\ref{eq:213}) prove (\ref{eq:207}) with $f$ replaced by $P$, and this concludes the proof. 
\end{proof}

Next we prove a Jensen-type formula that we will use to connect the logarithmic indicator function, $\ell_f$, to the distribution of zeros in logarithmic regions. Fix large $\varphi_0, u_0$ so that $\theta_{\varphi}(u)$ is well-defined in $\varphi \leq \varphi_0$, $u > u_0$. For $\varphi \in [0, \varphi_0]$ and $u > u_0$, set
\begin{align*}
    x(u, \varphi) := &\; u\cos{\theta_{\varphi}(u)} = \sqrt{u^2 - \varphi^2 \log^2{(1 + u})},\\
    y(u, \varphi) := &\; u\sin{\theta_{\varphi}(u)} = \varphi \log{(1 + u)}, 
\end{align*}
and note that we can use the Cauchy-Riemann equations to calculate that 
\begin{align*}
    \frac{\partial }{\partial u} \left[\arg f\left(ue^{i\theta_{\varphi}(u)} \right) \right] =&\; p(u, \varphi)\frac{\partial}{\partial u} \left[\log{\abs{f\left(ue^{i \theta_{\varphi}(u)} \right)}} \right] + q(u, \varphi)\frac{\partial}{\partial \varphi} \left[\log{\abs{f\left(ue^{i \theta_{\varphi}(u)} \right)}} \right]\\
    \frac{\partial }{\partial \varphi} \left[\arg f\left(ue^{i\theta_{\varphi}(u)} \right) \right] =&\; m(u, \varphi)\frac{\partial}{\partial u} \left[\log{\abs{f\left(ue^{i \theta_{\varphi}(u)} \right)}} \right] - p(u, \varphi)\frac{\partial}{\partial \varphi} \left[\log{\abs{f\left(ue^{i \theta_{\varphi}(u)} \right)}} \right] \tag{2.14} \label{eq:214}
\end{align*}
where 
\begin{equation*}
    p = \frac{y_{\varphi}y_{u} + x_{\varphi}x_{u}}{x_uy_\varphi - y_ux_{\varphi}}, \quad q = \frac{x_u^2 + y_u^2}{x_{\varphi}y_u - x_u y_{\varphi}}, \quad m = \frac{x_{\varphi}^2 + y_{\varphi}^2}{x_u y_\varphi - x_\varphi y_u} \tag{2.15} \label{eq:215}
\end{equation*}
(subscripts here denoting partial derivatives). We note that $q(u, \varphi) \sim \frac{1}{\log(1+u)}$ and therefore we expect 
\begin{equation*}
    J^r(\varphi) := \int_{r_0}^r q(u, \varphi)\log{\abs{f(ue^{i \theta_{\varphi}(u)})}}du \tag{2.16} \label{eq:216}
\end{equation*}
to have similarities with $\ell_f(\varphi)$. Finally, for $0 \leq t < \infty$, we recall that 
\begin{equation*}
    n(t; r) := \# \set{z \colon f(z) = 0, \operatorname{Re} z>0, 0 \geq \operatorname{Im}z \geq -t\log(1 + \abs{z}), \abs{z} < r} 
\end{equation*}
counting all zeros with multiplicity. 

We may now state the following proposition. One might compare this formula to \cite[(3.04)]{17} and \cite[Theorem 3]{2}. 

\begin{prop}\label{prop:2.4}
    For $f \in \mathcal F$, the following formula holds: 
    \begin{equation*}
        \int_{0}^r \frac{n(t; \nu)}{\nu} d\nu = \frac{1}{2\pi}
 \frac{\partial}{\partial \varphi} \left[\int_{r_0}^r \frac{J^\nu(\varphi)}{\nu}d\nu\right]_{\varphi = t} - \frac{1}{2\pi} \int_0^t \frac{m(r, \varphi)}{r} \log \abs{f(re^{i\theta_{\varphi}(u)})}d\varphi + o(r^\rho) \tag{2.17} \label{eq:217}
 \end{equation*}
 as $r \to \infty$, uniformly for $t$ in any bounded interval. 
\end{prop}

\begin{proof}
    Fix $t_0 > t$, and then take a large $r_0$ such that $\theta_{t}(r)$ is well-defined in $[0, t_0] \times [r_0, \infty)$. Applying the argument principle, we obtain
    \begin{align*}
        2\pi n(t; r) = 2\pi n(t; r_0) + &\; \int_0^t \frac{\partial \left[\arg f \left(r_0 e^{i \theta_{\varphi}(r_0)} \right) \right]}{\partial \varphi} d \varphi +  \int_{r_0}^r \frac{\partial \left[\arg f \left(u e^{i \theta_{t}(u)} \right) \right]}{\partial u} du \\
        -&\;  \int_0^t \frac{\partial \left[\arg f \left(r e^{i \theta_{\varphi}(r)} \right) \right]}{\partial \varphi} d \varphi - \int_{r_0}^r \frac{\partial \left[\arg f(u) \right]}{\partial u} du.
    \end{align*}
    The first two terms on the right are $O(1)$, uniformly for $t \leq t_0$, while using (\ref{eq:201}), the last term on the right is $O(r^{\rho -1})$, again uniformly in $t \leq t_0$. Using (\ref{eq:214}) on the other two terms, integrating by parts, and using that $p_{\varphi} - r\partial_r \left( \frac{m}{r} \right) = 0$, we get 
    \begin{align*}
        2\pi n(t; r) = \frac{\partial}{\partial \varphi}J^r(\varphi)\bigg\rvert_{\varphi = t} -&\; \int_{r_0}^r \left(p_u(u, t) + q_\varphi(u, t) \right) \log{\abs{f\left(ue^{i\theta_t(u)} \right)}}du\\
        +&\; 2p(r, t)\log{\abs{f\left(re^{i\theta_t(r)} \right)}} - r \frac{\partial}{\partial r} \left[\int_0^t \frac{m(r, \varphi)}{r} \log \abs{f\left( re^{t \theta_\varphi (r)}\right)} d\varphi \right] + O(r^{\rho-1}),
    \end{align*}
    uniformly in $t \leq t_0$. Divide by $2\pi r$ and integrate to find  
    \begin{align*}
        \int_0^r \frac{n(t; \nu)}{\nu} =&\; \frac{1}{2\pi} \frac{\partial}{\partial \varphi} \left[\int_{r_0}^r \frac{J^\nu(\varphi)}{\nu}d\nu \right]_{\varphi = t} - \frac{1}{2\pi} \int_{r_0}^r \frac{1}{\nu} \int_{r_0}^\nu \left(p_u(u, t) + q_\varphi(u, t) \right)\log{\abs{f\left(ue^{i\theta_t(u)} \right)}} dud\nu\\
        +&\; \frac{1}{\pi} \int_{r_0}^r \frac{p(\nu, t)}{\nu} \log{\abs{f\left(re^{i\theta_t(\nu)} \right)}}d\nu - \frac{1}{2\pi} \int_{0}^t \frac{m(r, \varphi)}{r} \log{\abs{f\left(re^{i\theta_\varphi(r)} \right)}}d \varphi + O\left(r^{\rho-1}\right)
    \end{align*}
    uniformly in $t \leq t_0$. Calculating from (\ref{eq:215}) that 
    \begin{align*}
        p_u(u, t) + q_\varphi(u, t) \sim&\; -\frac{2t}{u^2}\\
        p(u, t) \sim&\; \frac{t\log\left(1 + u \right)}{u},
    \end{align*}
    we can apply Lemma \ref{lem:2.3} to bound the second and third terms on the right to conclude the proof. 
\end{proof}

We now use the formula of Proposition \ref{prop:2.4} to prove Theorem \ref{thm:1.2} Recall that this theorem says that for any $f\in\mathcal{F}$ and $t\geq 0$, we have 

\begin{equation*}
    \varlimsup_{r\to\infty} \frac{1}{r^\rho} \int_0^t\int_0^r \frac{n(\tau, \nu)}{\nu} d\nu d\tau\leq \frac{\ell_f(t)}{2\pi\rho^2} \tag{2.18} \label{eq:218}
\end{equation*}



\begin{proof}[Proof of Theorem \ref{thm:1.2}]
    We begin by integrating (\ref{eq:217}) in $t$, and using that 
    \begin{equation*}
        \int_{r_0}^r \frac{J^\nu(0)}{\nu}d \nu = O\left(r^{\rho- 1} \right)
    \end{equation*}
    to get that 
    \begin{equation*}
        \int_0^t \int_0^r \frac{n(\tau; \nu)}{\nu} d\nu d\tau = \frac{1}{2\pi}\int_{r_0}^r \frac{J^\nu(t)}{\nu} d\nu - \frac{1}{2\pi} \int_0^t \int_0^\tau \frac{m(r, \varphi)}{r} \log{\abs{f\left(re^{i\theta_\varphi(r)} \right)}} d\varphi d\tau + o\left(r^\rho\right) \tag{2.19} \label{eq:219}
    \end{equation*}
    uniformly for $t$ in any bounded interval.

    Next, since $q(u, \varphi) \sim \frac{1}{\log\left(1 + u \right)}$ uniformly for bounded $\varphi$, we obtain from (\ref{eq:216}) that 
    \begin{equation*}
        \varlimsup_{r \to \infty} \frac{1}{2\pi r^\rho} \int_{r_0}^r \frac{J^\nu(t)}{\nu} d\nu \leq \frac{\ell_f(t)}{2\pi \rho^2}.
    \end{equation*}
    Calculating from (\ref{eq:215}) that $m(r, \varphi) \sim -\log\left(1 + r\right)$, the remaining term on the right in (\ref{eq:219}) is easily bounded above by $O\left(r^{\rho-1} \right)$ and the proof is complete. 
\end{proof}

As mentioned in the introduction, our applications to the study of scattering resonances in later sections will show that this theorem is sharp, in the sense that there exist non-trivial $f$ for which (\ref{eq:218}) is an equality for all $t \geq 0$. We can also obtain a bound on the counting function itself, rather than the integrated version in (\ref{eq:218}):

\begin{cor}\label{cor:2.5}
    For $f \in \mathcal F$, we have 
    \begin{equation*}
        \varlimsup_{r \to \infty} \frac{n\left(t; r\right)}{r^\rho} \leq \frac{e}{\pi \rho} \ell_f'\left(2t+\right)
    \end{equation*}
    for all $t \geq 0$.
\end{cor}
\begin{proof}
    Using Theorem \ref{thm:1.2}, we have for any $c, d > 1$, 
    \begin{equation*}
        \left(c-1\right)t\frac{\log d}{d^\rho} \varlimsup_{r \to \infty} \frac{n\left(t;r \right)}{r^\rho} \leq \varlimsup_{r\to \infty} \frac{1}{d^\rho r^\rho}\int_{t}^{ct}\int_r^{dr} \frac{n\left( \tau;  u\right)}{u} du d \tau \leq \frac{\ell_f(ct)}{2\pi \rho^2}
    \end{equation*}
    or 
    \begin{equation*}
        \varlimsup_{r\to \infty} \frac{n(t; r)}{r^\rho} \leq \frac{1}{2\pi \rho^2} \cdot \frac{c}{c-1} \cdot \frac{d^\rho}{\log d} \cdot \frac{\ell_f(ct)}{ct}.
    \end{equation*}
    Now, $\frac{d^\rho}{\log d}$ is minimized at $d = e^{1/\rho}$, where it is equal to $e\rho$. On the other hand, setting $c = 2$, and using the fact that $\ell_f$ is convex, $\ell_f(0) = 0$ (Lemma \ref{lem:2.1}), to obtain $\frac{\ell_f(2t)}{2t} \leq \ell_f'(2t+)$ for all $t \geq 0$, we complete the proof of the corollary. 
\end{proof}

Next we prove Theorem \ref{thm:1.3} showing that more information about the growth of $f$ along some logarithmic curves leads to a much stronger statement about the density of zeros. This theorem is analogous to \cite[Ch. III, Section 3, Theorem 3]{17} and \cite[Theorem 4]{2}, but our growth condition (see \eqref{eq:104} or \eqref{eq:220} below) is stronger than the ones assumed there. Recall that the theorem says that if for $f\in\mathcal{F}$ and some open interval $I\subset (0,\infty)$ the limit 

\begin{equation*}
    \lim_{r\to\infty} \frac{\log\abs{f(re^{i\theta_t(r)})}}{r^{\rho-1}\log r}=\ell_f(t) \tag{2.20} \label{eq:220}
\end{equation*}
exists for all $t\in I$, then 
\[
    \lim_{r\to\infty} \frac{n(t;r)}{r^\rho}=\frac{\ell_f'(t)}{2\pi\rho}
\]
for all $t\in I$ for which $\ell_f'(t)$ exists. 



\begin{proof}[Proof of Theorem \ref{thm:1.3}]
    By Proposition \ref{prop:2.4}, Lemma \ref{lem:2.2}, and the fact that $m(r, \varphi) \sim -\log(1 + r)$, we see that there is a set $E \subset (0, \infty)$ of finite measure such that 
    \begin{equation*}
        \int_0^r \frac{n(t, \nu)}{\nu}d\nu = \frac{1}{2\pi} \frac{\partial}{\partial \varphi} \left[\int_{r_0}^r \frac{J^\nu(\varphi)}{\nu} d\nu \right]_{\varphi=t} + o(r^\rho) \quad r \notin E \tag{2.21} \label{eq:221}
    \end{equation*}
    uniformly for $t$ in any bounded interval. For small $\abs{\delta}$, we integrate (\ref{eq:221}) to find 
    \begin{equation*}
        \frac{1}{\delta} \int_{t}^{t + \delta} \int_0^r \frac{n(\tau; \nu)}{\nu} d\nu d\tau = \frac{1}{2\pi} \int_{r_0}^r \frac{J^\nu(t + \delta)- J^\nu(t)}{\delta\nu}d\nu + o(r^\rho) \quad r \notin E. \tag{2.22} \label{eq:222}
    \end{equation*}

    We now let $t \in I$ be such that $\ell_f'(t)$ exists. Given $\varepsilon > 0$, take $\abs{\delta}$ small enough that $t + \delta \in I$ and 
    \begin{equation*}
        \abs{\frac{\ell_f(t + \delta) - \ell_f(t)}{\delta} - \ell_f'(t)} < \varepsilon. \tag{2.23} \label{eq:223}
    \end{equation*}
    Existence of the limit (\ref{eq:220}) implies the existence of the limit
    \begin{equation*}
        \lim_{r \to \infty} \frac{1}{r^\rho} \int_{r_0}^r \frac{J^\nu(t)}{\nu} d\nu = \frac{\ell_f(t)}{\rho^2},
    \end{equation*}
    and it follows from (\ref{eq:223}) that if $r > r_\varepsilon$ and $\abs{\delta}$ is small, then 
    \begin{equation*}
        \abs{\frac{1}{r^\rho} \int_{r_0}^r \frac{J^\nu(t + \delta) - J^\nu(t)}{\delta \nu} d\nu - \frac{\ell_f'(t)}{\rho^2}} < 2\epsilon.
    \end{equation*}
    Thus from (\ref{eq:222}), if $r > r_\varepsilon, r \notin E$, and $\abs{\delta}$ is small, then 
    \begin{equation*}
        \abs{\frac{1}{\delta r^\rho} \int_t^{t + \delta} \int_0^r \frac{n(\tau; \nu)}{\nu}d\nu d\tau - \frac{\ell_f'(t)}{2\pi \rho^2}} < 3\varepsilon.
    \end{equation*}
    
    Now, if $\delta < 0$, then 
    \begin{equation*}
        \frac{1}{r^\rho} \int_0^r \frac{n(t; \nu)}{\nu} d\nu \geq \frac{1}{\delta r^\rho} \int_t^{t + \delta}\int_0^r \frac{n(\tau; \nu)}{\nu} d\nu d\tau > \frac{\ell_f'(t)}{2\pi \rho^2} - 3\varepsilon,
    \end{equation*}
    and taking $\delta > 0$, we similarly obtain 
    \begin{equation*}
        \frac{1}{r^\rho} \int_0^r \frac{n(t; \nu)}{\nu}d\nu < \frac{\ell_f'(t)}{2\pi \rho^2} + 3\varepsilon,
    \end{equation*}
    both for $r > r_\varepsilon, r \notin E$. Since $\varepsilon$ is arbitrary, we've shown that 
    \begin{equation*}
        \lim_{\substack{r \to \infty \\ r \notin E}} \frac{1}{r^\rho} \int_0^r \frac{n(t; \nu)}{\nu} d\nu = \frac{\ell_f'(t)}{2\pi \rho^2}.
    \end{equation*}
    The same argument used at the end of the proof of \cite[Ch. III, Section 3, Theorem 3]{17} takes this to the conclusion of the theorem. 
    
\end{proof}

We conclude this section by explaining how these results apply to the scattering determinant, $\det S_V(-\lambda)$, for $V \in L_c^\infty(\R^n, \C), n \geq 1$ odd. We first note that by \cite[Theorem 3.54]{10}, $\det S_V(-\lambda)$ is meromorphic in $\C$, with at most finitely many poles in $\operatorname{Im} \lambda \leq 0$, say $\lambda_1, \lambda_2, \ldots, \lambda_m$ (where we repeat poles to the multiplicity). If we set 
\begin{equation*}
    f(\lambda) := \prod_{j=1}^m \frac{\lambda - \lambda_j}{\lambda + \lambda_j} \det S_V(-\lambda),
\end{equation*}
then $f$ is holomorphic in $\operatorname{Im} \lambda \leq 0$, and we claim that $f \in \mathcal F$. 

Indeed, for $n = 1$, that $f$ is of order 1 and finite type in $\operatorname{Im} \lambda \leq 0$ follows from \cite[(2.5.13)]{10}. For $n \geq 3$, that $f$ is of order at most $n$ and finite type follows from \cite[Theorem 5]{24} (see the remarks following \cite[Theorem 3.2]{6} for an explanation that \cite[Theorem 5]{24} also applies to complex-valued $V$). That (\ref{eq:201}) holds for $n \geq 3$ follows exactly as in \cite[Section 4]{6}. For $n =1$, we use Part 3 of \cite[Proof of Theorem 2.20]{10} (which makes no use of the hypothesis that $V$ be real-valued), to see that for $\lambda$ real, 
\begin{equation*}
    \frac{\frac{d}{d\lambda}\left[\det S_V(-\lambda) \right]}{\det S_V(-\lambda)} = -\operatorname{tr}\left(\partial_\lambda S_V(-\lambda) S_V(-\lambda)^{-1} \right) = O(\lambda^{-2}). 
\end{equation*}
Then (\ref{eq:201}) for $n = 1$ again follows as in \cite[Section 4]{6}.

Thus, we have verified that $f \in \mathcal F$; as mentioned at the beginning of the section, since the resonances and the zeros of $\det S_V(-\lambda)$ agree with multiplicity but for at most finitely many exceptions, this completes the proof of Theorem 1.4.

Finally, we remark that it may be possible to adapt the hypotheses in Definition \ref{def:1.1} so that the results apply more generally to black box scattering matrices (see \cite{23} and \cite[Section 4]{10}) in any dimension. This would require, at a minimum, an appropriate alternative to Definition \ref{def:1.1} (ii), and then our results would require an additional term involving the asymptotics of the scattering phase, which has been described in \cite{5}.

\section{MANY STRINGS OF RESONANCES FOR SINGULAR POTENTIALS}\label{sec:3}

\subsection{Preliminaries}\label{sec:4.1}
In this section we will study the class $\mathcal V$ of singular potentials in dimension one from Definition \ref{def:1.5} in the introduction, and prove Theorems \ref{thm:1.6}, \ref{thm:1.7}, and \ref{thm:1.8}. We begin with some technical material generalizing that of \cite[Section 2 and Section 3.1]{7}, and we start with the representation of the scattering matrix used there (see also \cite{11}, \cite{22}, \cite[Chapter 2]{10}). 

For $V \in L_c^\infty(\R; \C)$ let $R_V(\lambda) := \left( -\frac{d^2}{dx^2} + V - \lambda^2 \right)^{-1} \colon L_c^2 \to H_{\operatorname{loc}}^2$ denote the meromorphic continuation of the resolvent from the physical half-plane $\set{\operatorname{Im} \gg 1}$ to all of $\C$. We note in particular that the free resolvent, $R_0$, has explicit integral kernel 
\begin{equation*}
    R_0(\lambda; x,y) = \frac{i}{2\lambda} e^{i \lambda \abs{x - y}}. \tag{3.1} \label{eq:401}
\end{equation*}

Define the functions 
\begin{equation*}
    f_{\pm}^V(x, \lambda) := e^{\pm i \lambda x} R_V(-\lambda)e^{\mp i \lambda \cdot} V.
\end{equation*}
By \cite[Lemma 3.3]{11}, we have 
\begin{equation*}
    \abs{f_\pm^V(\cdot, \lambda)} \leq \frac{C}{\abs{\lambda}} \tag{3.2} \label{eq:402}
\end{equation*}
for $\operatorname{Im} \lambda \leq 0, \abs{\lambda} > C_1$. The scattering matrix, $S_V$, then has the following representation:
\begin{equation*}
    S_V(\lambda) = I + \frac{1}{2i\lambda} \begin{pmatrix}
        T_+(\lambda) & \rho_-(\lambda) \\  \rho_+(\lambda) & T_-(\lambda)
    \end{pmatrix} \tag{3.3} \label{eq:403}
\end{equation*}
where 
\begin{equation*}
    \rho_\pm(-\lambda) := \int e^{\mp 2 i \lambda x} V(x) (1 - f_\pm^V(x, \lambda)) dx \quad T_\pm(-\lambda) := \int V(x)(1 - f_\pm^V(x, \lambda))dx. \tag{3.4} \label{eq:404}
\end{equation*}

\noindent Note that $\frac{1}{2i\lambda} \rho_\pm (\lambda)$ are the reflection coefficients, and $1 + \frac{1}{2i \lambda} T_+(\lambda) = 1 + \frac{1}{2i \lambda} T_-(\lambda)$ is the transmission coefficient. 

Directly from (\ref{eq:402}) and (\ref{eq:403}), we have an expansion for the scattering determinant 
\begin{equation*}
    \det S_V(-\lambda) = 1 + \frac{1}{4\lambda^2} \rho_-(-\lambda) \rho_+(-\lambda) + O(\abs{\lambda}^{-1}). \tag{3.5} \label{eq:405}
\end{equation*}

We now define the following class of functions:
\begin{definition}\label{def:4.1}
    Let $x_0 < x_1 < \cdots < x_N$ and let $\set{\nu_k}_{k=0}^N \subset \N_0$. We say that a function $U \in \mathcal U(\set{x_k}, \set{\nu_k})$ provided $U \in C^\infty (\R/[x_0, x_N])$ and 
    \begin{enumerate}[label=\roman*)] 
        \item $U \in C^{M_k}([x_{k-1}, x_k])$ where $M_k := \max \set{\nu_{k-1}, \nu_k} + 1$, $k = 1, 2, \ldots , N$,
        
        \item $U^{(\nu)}(x_k-) = U^{(\nu)}(x_k +)$ for all $\nu < \nu_k$, $k = 0, 1, \ldots , N$. 
    \end{enumerate}
\end{definition}

Clearly the definition covers the class $\mathcal V$, but has added generality so as to also include the functions $f_\pm^V$ for any $V \in \mathcal V$. The following lemma is similar to \cite[Lemma 3.1]{7}; as expected, the statement and proof are a bit more complicated with the level of generality we consider. In fact, to state it adequately, we need a few more definitions. 

Fix $k \in \set{1, 2, \ldots , N}$ and set 
\begin{align*}
    X_{k, 0} :=&\; \set{x_0, x_1, \ldots , x_{k-1}},\\
    S_{k, 0} := &\; \set{\nu_0, \nu_1, \ldots, \nu_{k-1}},\\
    D_{k, 0} := &\; \min S_{k, 0},
\end{align*}
and for $m = 1, 2, \ldots$, 
\begin{align*}
    X_{k, m} := &\; \set{x_p \in X_{k, 0} \colon \nu_p = D_{k, m-1}},\\
    S_{k, m} := &\; \set{\nu \in S_{k, 0} \colon \nu > D_{k, m-1}},\\
    D_{k, m} := &\; \min S_{k, m},
\end{align*}
the latter definitions terminating after a finite number of steps. Finally, with $M_k$ as in Definition \ref{def:4.1}, we either have $1 + D_{k, 0} = M_k$ or else $2 + D_{k,0} \leq M_k$. In the latter case we define $m_k \geq 1$ to be the largest $m$ for which $2 + D_{k, m-1} \leq M_k$. Note that this implies $2 + D_{k, m_k - 1} \leq M_k \leq 1 + D_{k, m_k}$. 

We may finally state our main lemma; there is an analogue for $f_+^V$:
\begin{lm}\label{lem:4.2}
    Let $V \in \mathcal U(\set{x_k}, \set{\nu_k})$ with $\operatorname{ch supp} V = [x_0, x_N]$. Then there exists $C_0 > 0$ so that in $\operatorname{Im}\lambda \leq 0, \abs{\lambda} > C_0$, we have $f_-(\cdot, \lambda) = f_-^V( \cdot, \lambda) \in \mathcal U(\set{x_k}, \set{\nu_k + 2})$, and the following expansions hold: In each $[x_{k-1}, x_k], k = 1, 2, \ldots, N$, we have 
    \begin{enumerate}[label=\roman*)] 
        \item $\partial_x^d f_-(\cdot, \lambda) = O(\abs{\lambda}^{-1})$ for $0 \leq d \leq 1 + D_{k, 0}$,
        
        \item For $\mu = 1, 2, \ldots , m_k$, 
        \begin{equation*}
            \partial_x^d f_-(\cdot, \lambda) = \sum_{m=1}^\mu \sum_{x_l \in X_{k, m}} e^{-2i \lambda(x - x_l)} O_{l, m}\left(\abs{\lambda}^{d - D_{k, m-1} -2} \right) + O(\abs{\lambda}^{-1})
        \end{equation*}
        for $2 + D_{k, \mu-1} \leq d \leq \min \set{1 + D_{k, \mu}, M_k}$, with the understanding that if $M_k = 1 + D_{k, 0}$, we have only the bounds in (i). 
    \end{enumerate}
    
\end{lm}

\begin{proof}
    Standard elliptic regularity theory and the definition of $f_-$ in terms of the resolvent show that $f_- \in \mathcal U(\set{x_k}, \set{\nu_k + 2})$. Using a resolvent identity, we get 
    \begin{equation*}
        f_-(x, \lambda) = e^{-i \lambda x}R_0(-\lambda)(e^{i \lambda \cdot}V)(x)-e^{-i \lambda x}R_0(-\lambda)(e^{i \lambda \cdot} V f_-(\cdot, \lambda))(x).
    \end{equation*}
    From the explicit integral kernel of the free resolvent, (\ref{eq:401}), we obtain the equation 
    \begin{equation*}
        \partial_x f_-(x, \lambda) = - \int_{x_0}^x e^{-2i\lambda(x - y)}V(y)dy + \int_{x_0}^x e^{-2i \lambda (x-y)} V(y)f_-(y, \lambda) dy, \tag{3.6} \label{eq:406}
    \end{equation*}
    which immediately shows that any bounds made in an interval $[x_{k-1}, x_k]$ will depend upon the behavior of $f_-$ and $V$ on each of the intervals $[x_0, x_1], [x_1, x_2], \ldots$, to the left of $[x_{k-1}, x_k]$. Accordingly, we first prove the portion of the lemma involving $[x_0, x_1]$. 

    If $\nu_1 \leq \nu_0$, then $1 + D_{1, 0} = M_1$, and an elementary induction proof using (\ref{eq:401}) and beginning from (\ref{eq:402}) shows that 
    \begin{equation*}
        \partial_x^df_-(\cdot, \lambda) = O(\abs{\lambda}^{-1}) \quad \text{for } 0 \leq d \leq M_1.
    \end{equation*}
    If $\nu_0 < \nu_1$, then just as above, we easily obtain part (i) in this case. We then differentiate (\ref{eq:406}) and integrate by parts to find 
    \begin{equation*}
        \partial_x^{1 + D_{1, 0}} f_-(x, \lambda) = - \int_{x_0}^x e^{-2i \lambda(x-y)} V^{(D_{1, 0})}(y)dy + \int_{x_0}^x e^{-2i \lambda(x-y)}(Vf_-)^{(D_{1, 0})}(y)dy,
    \end{equation*}
    whence 
    \begin{align*}
        \partial_x^{2 + D_{1, 0}}f_-(x, \lambda) =&\; -V^{(D_{1, 0})}(x_0 +)e^{-2i \lambda(x-x_0)} - \int_{x_0}^x e^{-2i \lambda(x-y)} V^{(D_{1, 0} + 1)}(y) dy \\
        +&\; (Vf_-)^{(D_{1, 0})}(x_0+, \lambda)e^{-2i \lambda(x - x_0)} + \int_{x_0}^x e^{-2i \lambda(x-y)}(Vf_-)^{(D_{1, 0} + 1)}(y)dy\\
        =&\; e^{-2i \lambda(x-x_0)}O(1) + O(\abs{\lambda}^{-1}).
    \end{align*}
    Continuing this way, using the previously made bounds at each step, we obtain 
    \begin{equation*}
        \partial_x^d f_-(x, \lambda) = e^{-2i \lambda(x - x_0)}O(\abs{\lambda}^{d - D_{1, 0} - 2}) + O(\abs{\lambda}^{-1}) \quad \text{for } 2 + D_{1, 0} \leq d \leq M_1,
    \end{equation*}
    concluding the proof of the lemma for the interval $[x_0, x_1]$. 

    Next, we consider an interval $[x_{k-1}, x_k]$, assuming the lemma proved for all intervals to the left of it. The proof of (i) for this interval is again very easy, using (\ref{eq:406}) and beginning from (\ref{eq:402}). If $M_k = 1 + D_{k, 0}$, we are done. Otherwise, since $V\in C^{(D_{k, 0})}([x_0, x_k])$, we get 
    \begin{align*}
        \partial_x^{1 + D_{k, 0}}f_-(x, \lambda) =&\; - \int_{x_0}^x e^{-2i \lambda(x-y)}V^{(D_{k, 0})}(y)dy + \int_{x_0}^x e^{-2i \lambda(x-y)}(Vf_-)^{(D_{k, 0})}(y)dy \\
        =&\; - \sum_{p=1}^{k-1}\int_{x_{p-1}}^{x_p} e^{-2i \lambda(x-y)} V^{(D_{k, 0})}(y)dy - \int_{x_{k-1}}^x e^{-2i \lambda(x-y)}V^{(D_{k, 0})}(y)dy \\
        +&\; \sum_{p=1}^{k-1} \int_{x_{p-1}}^{x_p} e^{-2i \lambda(x-y)}(Vf_-)^{(D_{k, 0})}(y)dy + \int_{x_{k-1}}^x e^{-2i \lambda(x-y)}(Vf_-)^{(D_{k, 0})}(y)dy.
    \end{align*}
    Differentiate and then integrate by parts to get 
    \begin{align*}
        \partial_x^{2 + D_{k, 0}}f_-(x, \lambda) =&\; - V^{(D_{k, 0})}(x_0 +)e^{-2i \lambda(x - x_0)} + \sum_{p=1}^{k-1}\left(V^{(D_{k, 0})}(x_p -) - V^{(D_{k, 0})}(x_p+) \right)e^{-2i\lambda(x - x_p)}\\
        &\; -\sum_{p=1}^{k-1} \int_{x_{p-1}}^{x_p} e^{-2i \lambda(x-y)}V^{(D_{k, 0} + 1)}(y)dy - \int_{x_{k-1}}^x e^{-2i \lambda(x-y)}V^{(D_{k, 0} + 1)}(y) dy \\
        &\; +(Vf_-)^{(D_{k, 0})}(x_0+)e^{-2i \lambda(x - x_0)} - \sum_{p=1}^{k-1} \left((Vf_-)^{(D_{k, 0})}(x_p-) - (Vf_-)^{(D_{k, 0})}(x_p+) \right)e^{-2i \lambda(x - x_p)}\\
        +&\; \sum_{p=1}^{k-1} \int_{x_{p-1}}^{x_p}e^{-2i \lambda(x-y)} (Vf_-)^{(D_{k, 0} + 1)}(y)dy + \int_{x_{k-1}}^x e^{-2i \lambda(x-y)}(Vf_-)^{(D_{k, 0} + 1)}(y)dy.
    \end{align*}
    But $V^{(D_{k, 0})}(x_0+)= 0$ unless $\nu_0 = D_{k, 0}$, and $V^{(D_{k, 0})}(x_p -) = V^{(D_{k, 0})}(x_p +)$ unless $\nu_p = D_{k, 0}$ so using the induction hypothesis to bound the integral terms, one easily produces from this the case $\mu = 1$ of (ii) in the lemma. 

    The remaining bounds are produced in exactly the same way, and we conclude the proof. 
\end{proof}

In the same way as in \cite[Section 3.1]{7}, we use this to obtain asymptotic information on the reflection coefficients $\rho_\pm$

\begin{lm}\label{lem:4.3}
    For $V \in \mathcal U(\set{x_k}, \set{\nu_k})$ with $\operatorname{supp} V = [x_0, x_N]$ we have 
    \begin{equation*}
        \rho_\pm (-\lambda) = \widehat{V}(\pm 2\lambda) + \sum_{p = 0}^N e^{\mp 2i \lambda x_p} O(\abs{\lambda}^{-\nu_p - 2}).
    \end{equation*}
\end{lm}
\begin{proof}
    The definition of $\rho_-$ in (\ref{eq:404}) gives 
    \begin{align*}
        \rho_-(-\lambda) =&\; \widehat{V}(-2\lambda) - \int_{x_0}^{x_N} e^{2i\lambda x} V(x)f_-(x, \lambda) dx\\
        =&\; \widehat{V}(-2\lambda) - \sum_{p = 0}^{N-1} \int_{x_p}^{x_{p+1}} e^{2i\lambda x} V(x)f_-(x, \lambda)dx.
    \end{align*}
    Integrating by parts an appropriate number of times in each term, noting cancellations due to regularity of the integrand across the $x_p$, and then appealing to Lemma \ref{lem:4.2}, we obtain the desired expansion for $\rho_-$. An analogous proof gives the expansion for $\rho_+$. 
\end{proof}

\subsection{Proofs of Theorems \ref{thm:1.6} and \ref{thm:1.7}}\label{sec:4.2}
A straightforward integration by parts argument allows us to show that for $V \in \mathcal V$, we have 
\begin{equation*}
    \widehat V(\pm 2\lambda) = \sum_{p=0}^N \frac{V^{(\nu_p)}(x_p+) - V^{(\nu_p)}(x_p-)}{(\pm 2i \lambda)^{\nu_p + 1}}e^{\mp 2i \lambda x_p}(1 + o(1))
\end{equation*}
in $\operatorname{Im} \lambda \leq 0$. Using this, Lemma \ref{lem:4.3}, and the expansion of the scattering determinant, (\ref{eq:405}), we obtain that for $C_{p, m} := (-1)^{\nu_p + 1}(V^{(\nu_p)}(x_p+) - V^{(\nu_p)}(x_p-)) \cdot (V^{(\nu_m)}(x_m +) - V^{(\nu_m)}(x_m-))$, we have 
\begin{equation*}
    \det S_V(-\lambda) = 1 + \sum_{m=0}^N\sum_{p = m+1}^N \frac{C_{p, m}}{\lambda^{\nu_p + \nu_m + 4}}e^{2i \lambda(x_p - x_m)} (1 + o(1)) + O(\abs{\lambda}^{-1}). \tag{3.7} \label{eq:407}
\end{equation*}

\begin{remark}\label{remark:1}
    It is here that condition (iii) from Definition \ref{def:1.5} plays its role. If for some $T \geq 1$, the lengths of some subintervals $\set{[x_{m_t}, x_{p_t}]}_{t = 0}^T$ are all equal, and if also the corresponding $\set{\nu_{m_t} + \nu_{p_t}}_{t=0}^T$ are all equal, then without condition (iii) we \emph{could} have a term of the form 
    \begin{equation*}
        \frac{e^{2i \lambda (x_{p_0} - x_{m_0})}}{\lambda^{\nu_{p_0} + \nu_{m_0} + 4}} o(1)
    \end{equation*}
    in (\ref{eq:407}) and our analysis below would not apply. However, we note that even if this happened, if there were some interval $[x_{\overline m}, x_{\overline p}]$ with $x_{\overline p} - x_{\overline m} \geq x_{p_0} - x_{m_0}$ and $\nu_{\overline p} + \nu_{\overline m} < \nu_{p_0} + \nu_{m_0}$, then 
    \begin{equation*}
        \frac{e^{2i \lambda(x_{p_0} - x_{m_0})}}{\lambda^{\nu_{p_0} + \nu_{m_0} + 4}} =  \frac{e^{2i \lambda(x_{\overline p} - x_{\overline m})}}{\lambda^{\nu_{\overline p} + \nu_{\overline m} + 4}}o(1)
    \end{equation*}
    in $\operatorname{Im} \lambda \leq 0$, so the terms corresponding to the intervals $\set{[x_{m_t}, x_{p_t}]}_{t=0}^T$ \emph{absorb} into the one corresponding to $[x_{\overline m}, x_{\overline p}]$ and the requirement in (iii) is then unnecessary (see Section \ref{sec:4.3} for more on this).
\end{remark}

\begin{remark}\label{remark:2}
    For large enough $M> 0$, it's easy to see that 
    \begin{equation*}
        \abs{\lambda^{\nu_N + \nu_0 + 4} e^{-2i \lambda(x_N - x_0)}} = o(1)
    \end{equation*}
    and 
    \begin{equation*}
        C_{N, 0}^{-1} \lambda^{\nu_N + \nu_0 + 4}e^{-2i \lambda(x_N - x_0)} \sum_{m=0}^N \sum_{p = m+1}^N \frac{C_{p, m}}{\lambda^{\nu_p + \nu_m + 4}}e^{2i\lambda(x_p - x_m)}(1 + o(1)) = 1 + o(1)
    \end{equation*}
    as $\abs{\lambda} \to \infty$ through $\set{\operatorname{Im} \lambda \leq -M \log(1 + \abs{\lambda}), \operatorname{Re} \lambda >0}$. This shows that in the same region, 
    \begin{equation*}
        C_{N, 0}^{-1} \lambda^{\nu_N + \nu_0 + 4}e^{-2i \lambda(x_N - x_0)} \det S_V(-\lambda) = 1 + o(1)
    \end{equation*}
    so there are at most finitely many resonances in $\set{\operatorname{Im}\lambda \leq -M \log(1 + \abs{\lambda}), \operatorname{Re} \lambda > 0}$. This observation is used in the proof of Theorem \ref{thm:1.7}. 
\end{remark}

The next lemma implies that $\det S_V(- \lambda)$ has very nice asymptotic behavior along logarithmic curves. 

\begin{lm}\label{lem:4.4}
    Let $M \in \N, \alpha_m \geq 0, A_m \in \R, C_m \in \C/ \set{0}$, and 
    \begin{equation*}
        f(\lambda) = 1 + \sum_{m=1}^M \frac{C_m}{\lambda^{\alpha_m}} e^{2i \lambda A_m}(1 + o_m(1)) + o(1) \quad \text{in } \operatorname{Im} \lambda \leq 0, \operatorname{Re} \lambda > 1. \tag{3.8} \label{eq:308}
    \end{equation*}
    Then the limit 
    \begin{equation*}
        \lim_{r \to \infty} \frac{\log \abs{ f\left( re^{i \theta_t(r)} \right)}}{\log r} \tag{3.9} \label{eq:309}
    \end{equation*}
    exists for all but finitely many $t \in (0, \infty)$. 
\end{lm}
\begin{proof}

    We first note that if $A_m \leq 0$ for some $m$ in \eqref{eq:308}, then the corresponding term either decays along every $\log$ curve in the lower half-plane or is constant. In the latter case we divide \eqref{eq:308} by a constant to produce a new function of the form \eqref{eq:308} such that the limit \eqref{eq:309} exists for the new function if and only if it exists for the old one. Thus, we may assume the $A_m>0$. 

    The proof is by induction on $M$, and we begin with the case in which 
    \begin{equation*}
        f(\lambda) = 1 + \frac{C}{\lambda^\alpha}e^{2i \lambda A}(1 + o(1)) + o(1) \quad \text{in } \operatorname{Im} \lambda \leq 0, \operatorname{Re} \lambda > 1.
    \end{equation*}
    Set $T = \frac{\alpha}{2A}$ and notice that for $r \to \infty$ and $\lambda = re^{i \theta_t(r)}$, we have 
    \begin{equation*}
        \abs{\lambda^{-\alpha}e^{2i\lambda A}} = o(1) \quad \text{or} \quad \abs{\lambda^\alpha e^{-2i \lambda A}} = o(1) 
    \end{equation*}
    according to $t < T$ or $t > T$, respectively. In the first case the limit (\ref{eq:309}) is zero. In the second, we pull out the $C\lambda^{-\alpha} e^{2i \lambda A}$ term to write 
    \begin{equation*}
        f(\lambda) = \frac{C}{\lambda^\alpha} e^{2i \lambda A}(1 + o(1)),
    \end{equation*}
    whence 
    \begin{equation*}
        \log \abs{f(re^{i \theta_t(r)})} = 2At \log(1 + r) - \alpha \log r + \log \abs{C} + o(1)
    \end{equation*}
    so the limit (\ref{eq:309}) exists for all $t \neq T$, proving the case $M = 1$. 
    
    Assuming the case $M -1$ proven, for $f$ as in (\ref{eq:308}) we may assume that the term $\frac{C_M}{\lambda^{\alpha_M}}e^{2i\lambda A_M}$ is such that $\alpha_M = \min_m \set{\alpha_m}$. If more than one $\alpha_m$ is minimal, we also assume $A_M$ is maximal amongst those terms, and this uniquely selects the term $C_M\lambda^{-\alpha_M} e^{2i\lambda A_M}$. We write 
    \begin{equation*}
        f(\lambda) = g_{M - 1}(\lambda) + \frac{C_M}{\lambda^{\alpha_M}}e^{2i \lambda A_M}(1 + o(1)) \tag{3.10} \label{eq:310}
    \end{equation*}
    where $g_{M-1}(\lambda)$ satisfies the induction hypothesis. Set $T = \frac{\alpha_M}{2A_M}$. Similarly to part one, we have 
    \begin{equation*}
        \abs{\lambda^{-\alpha_M}e^{2i \lambda A_M}} = o(1) \quad \text{or} \quad \abs{\lambda^{\alpha_M}e^{-2i \lambda A_M}} = o(1)
    \end{equation*}
    as $r \to \infty$ with $\lambda = re^{i\theta_t(r)}$ according to $t < T$ or $t > T$. In the first case, we absorb the second term in (\ref{eq:310}) into $g_{M-1}$ and the limit (\ref{eq:309}) exists for all but finitely many $t < T$ by the induction hypothesis. In the second case, we write 
    \begin{equation*}
        f(\lambda) = \frac{C_M}{\lambda^{\alpha_M}}e^{2i \lambda A_M}\left(1 + C_M^{-1}\lambda^{\alpha_M} e^{-2i \lambda A_M}g_{M-1}(\lambda) + o(1) \right).
    \end{equation*}
    By our choice of the $\alpha_M, A_M$, the factor in the parentheses is a function satisfying the induction hypothesis, and therefore the limit (\ref{eq:309}) exists for all but finitely many $t > T$ as well. This proves the lemma. 
\end{proof}

\begin{remark}\label{rem:3}
    Lemma \ref{lem:4.4} and (\ref{eq:407}) show that 
    \begin{equation*}
        \lim_{r \to \infty} \frac{\log \abs{\det S_V(-re^{i \theta_t(r)})}}{\log r} = \ell_V(t) \tag{3.11} \label{eq:411}
    \end{equation*}
    exists for all but finitely many $t$. The proof above shows moreover that the $t$ for which the limit (\ref{eq:411}) does not exist must coincide with the discontinuities of $\ell_V'$. 
\end{remark}

As a corollary, we may now give the proof of Theorem \ref{thm:1.7}. 

\begin{proof}[Proof of Theorem \ref{thm:1.7}]
    Remark \ref{remark:2} after (\ref{eq:407}) shows that there exists $M > 0$ so that all but finitely many resonances in $\set{\operatorname{Re} \lambda \geq 0}$ belong to the region $\set{\operatorname{Im}\lambda \geq -M \log(1 + \abs{\lambda})} $. 

    Now let $t_1$, be a point at which $\ell_V'(t_1)$ exists. By the above considerations, the limit in (\ref{eq:411}) also exists, and it follows from the fact that discontinuities of $\ell_V'$ are isolated that there exists $\varepsilon > 0$ such that the limit (\ref{eq:411}) exists for all $t \in [t_1 - \varepsilon, t_1 + \varepsilon]$. In particular, we see that for any $t_1$ for which $\ell_V'(t_1)$ exists, there exists $\varepsilon > 0$ and $C > 0$ such that there are no resonances in the set 
    \begin{equation*}
        \set{\lambda \in \C \colon -(t_1 + \varepsilon) \log(1 + \abs{\lambda}) \leq \operatorname{Im} \lambda \leq -(t_1 - \varepsilon) \log(1 + \abs{\lambda}), \operatorname{Re} \lambda \geq 0, \abs{\lambda} > C}.
    \end{equation*}
    The conclusion of the theorem follows from these considerations. 
\end{proof}

The proof of Theorem \ref{thm:1.6} is also now very easy:

\begin{proof}[Proof of Theorem \ref{thm:1.6}]
By Remark \ref{rem:3} after Lemma \ref{lem:4.4}, and application of Theorem \ref{thm:1.3}, we have 
\begin{equation*}
    \lim_{r \to \infty} \frac{\# \set{\lambda_j \in \mathcal R_V: -s\log(1 + \abs{\lambda_j}) > \operatorname{Im} \lambda_j > -t\log(1 + \abs{\lambda_j}), \operatorname{Re} \lambda_j > 0, \abs{\lambda_j} < r}}{r} = \frac{1}{2\pi}\left(\ell_V'(t) - \ell_V'(s) \right)
\end{equation*}
for any $0 < s < t$ for which $\ell_V'(s), \ell_V'(t)$ exist; the theorem follows from this by a simple argument. 
    
\end{proof}

\subsection{Proof of Theorem \ref{thm:1.8}}\label{sec:4.3}
Having proven Theorems \ref{thm:1.6} and \ref{thm:1.7}, we now take a more careful look at the expansion in (\ref{eq:407}) to obtain more precise information about the connection between the parameters defining the singularities of $V$, discontinuities of $\ell_V'$, and the location and density of the strings of resonances, and to give the proof of Theorem \ref{thm:1.8}. 

In Remark \ref{remark:1} after the expansion (\ref{eq:407}) we noted a situation in which certain terms in (\ref{eq:407}) may absorb into a more dominant one. We now discuss this in a bit more detail. If for some pairs $(p_0, m_0), (p_1, m_1)$ with $0 \leq m_i < p_i \leq N$, we have $x_{p_0} - x_{m_0} = x_{p_1} - x_{m_1}$, then there are just three cases: 

\begin{enumerate}[label=\roman*)] 
        \item If $\nu_{p_0} + \nu_{m_0} + 4 < \nu_{p_1} + \nu_{m_1} + 4$, then 
        \begin{equation*}
            \lambda^{-\nu_{p_1} - \nu_{m_1} - 4}e^{2i\lambda(x_{p_1} - x_{m_1})} = \lambda^{-\nu_{p_0} - \nu_{m_0} - 4}e^{2i\lambda(x_{p_0} - x_{m_0})}o(1)
        \end{equation*}
        in $\operatorname{Im}\lambda \leq 0$, so we may absorb the term in (\ref{eq:407}) corresponding to $(p_1, m_1)$ into the one for $(p_0, m_0)$. 
        
        \item If $\nu_{p_0} + \nu_{m_0} + 4 = \nu_{p_1} + \nu_{m_1} + 4$, then we let $\set{(p_t, m_t)}_{t=0}^T$ be all the pairs with $x_{p_t} - x_{m_t} = x_{p_0} - x_{m_0}$ and $\nu_{p_t} + \nu_{m_t} = \nu_{p_0} + \nu_{m_0}$. Then the terms corresponding to $\set{(p_t, m_t)}_{t = 0}^T$ merge into the one for $(p_0, m_0)$, with a new constant $= \sum_{t=0}^T C_{p_t, m_t}$ (which is nonzero by condition (iii) of  Definition \ref{def:1.5}).

        \item If $\nu_{p_0} + \nu_{m_0} + 4 > \nu_{p_1} + \nu_{m_1} + 4$, we absorb the $(p_0, m_0)$ term into the $(p_1, m_1)$ term, similarly to (i). 
    \end{enumerate}
The point of this consideration is that in each case, there is actually only one contributing pair $(p_0, m_0)$ in the expansion (\ref{eq:407}) with the length $x_{p_0} - x_{m_0} = x_{p_1} - x_{m_1}$ in the exponential factor; clearly this argument also works for any finite number of pairs and we may therefore define the following equivalence classes of pairs:
\begin{equation*}
    [(p, m)] := \set{(p', m') \colon x_{p'} - x_{m'} = x_p - x_m \text{ and } \nu_{p'} + \nu_{m'} = \nu_p + \nu_m}. \tag{3.12} \label{eq:412}
\end{equation*}
We also define constants 
\begin{equation*}
    C_{[(p, m)]} := \sum_{(p', m') \in [(p, m)]} C_{p', m'}.
\end{equation*}
If we now let $S$ be a set of pairs containing one and only one representative from each equivalence class, (\ref{eq:412}), then we may rewrite (\ref{eq:407}) as 
\begin{equation*}
    \det S_V(-\lambda) = 1 + \sum_{(p, m) \in S} \frac{C_{[(p, m)]}}{\lambda^{\nu_p + \nu_m + 4}}e^{2i \lambda (x_p - x_m)}(1 + o(1)) + O(\abs{\lambda}^{-1}). \tag{3.13} \label{eq:413}
\end{equation*}

Having established this, we may now begin the proof of Theorem \ref{thm:1.8}. With $t_0 := \min_{0 \leq m < p \leq N} \frac{4 + \nu_p + \nu_m}{2(x_p - x_m)}$ as in the statement of Theorem \ref{thm:1.8}, we find that for any $0 \leq m < p \leq N, \lambda^{-(4 + \nu_p + \nu_m)}e^{2i\lambda(x_p - x_m)} = o(1)$ as $r \to \infty$ with $ \lambda = re^{i \theta_t(r)}$ and $t < t_0$. From (\ref{eq:413}), $\det S_V(-re^{i \theta_t(r)}) = 1 + o(1)$ there and hence $\ell_V(t) \equiv 0$ for $t < t_0$. 

Let $S_0$ denote the set of all pairs $(p, m) \in S$ with 
\begin{equation*}
    \frac{4 + \nu_p + \nu_m}{2(x_p - x_m)} = t_0
\end{equation*}
(there is at least one - and maybe only one - such pair). Let $(p_0, m_0) \in S_0$ be the pair such that $x_{p_0} - x_{m_0}$ (hence also $4 + \nu_{p_0} + \nu_{m_0}$) is maximized amongst the pairs in $S_0$ (here we used our preliminary considerations about uniqueness of these lengths). Using that for any other pair $(p, m) \in S_0$ (if there is one), we also have 
\begin{equation*}
    \frac{\nu_{p_0} + \nu_{m_0} - (\nu_p + \nu_m)}{2[(x_{p_0} - x_{m_0}) - (x_p - x_m)]} = t_0,
\end{equation*}
we see that as $r \to \infty$ with $\lambda = re^{i \theta_t(r)}$ and $t > t_0$, 
\begin{equation*}
    \abs{\lambda^{4 + \nu_{p_0} + \nu_{m_0}}e^{-2i\lambda (x_{p_0} - x_{m_0})}} = o(1)
\end{equation*}
and 
\begin{equation*}
    \abs{\lambda^{\nu_{p_0} + \nu_{m_0} - (\nu_p + \nu_m)}e^{-2i\lambda [(x_{p_0} - x_{m_0}) - (x_p - x_m)]}} = o(1).
\end{equation*}
From (\ref{eq:413}), we obtain that as $r \to \infty$ with $\lambda = re^{i \theta_t(r)}$ and $t > t_0$,
\begin{align*}
    C_{[(p_0, m_0)]}^{-1} \lambda^{4 + \nu_{p_0} + \nu_{m_0}} &e^{-2i \lambda(x_{p_0} - x_{m_0})}\det S_V(-\lambda)\\
    =&\; 1 + \sum_{(p, m) \in S/S_0} \frac{C_{[(p_0, m_0)]}^{-1}C_{[(p, m)]}}{\lambda^{\nu_p + \nu_m - (\nu_{p_0} + \nu_{m_0})}}e^{2i \lambda [(x_p - x_m) - (x_{p_0} - x_{m_0})]}(1 + o(1)) + o(1). \tag{3.14} \label{eq:414}
\end{align*}
Now, for $(p, m) \in S/S_0$, the fact that $(4 + \nu_{p_0} + \nu_{m_0})/2(x_{p_0} - x_{m_0}) < (4 + \nu_p + \nu_m)/2(x_p - x_m)$ implies that 
\begin{equation*}
    t_0 < \frac{\nu_p + \nu_m - (\nu_{p_0} + \nu_{m_0})}{2[(x_p - x_m) - (x_{p_0} - x_{m_0})]}.
\end{equation*}
If we define 
\begin{equation*}
    t_1 := \min_{(p, m) \in S/S_0} \frac{\nu_p + \nu_m - (\nu_{p_0} + \nu_{m_0})}{2[(x_p - x_m) - (x_{p_0} - x_{m_0})]}, \tag{3.15} \label{eq:415}
\end{equation*}
then it follow from (\ref{eq:414}) that as $r \to \infty$ with $\lambda = re^{i \theta_t(r)}$ and $t_0 < t < t_1$, we have 
\begin{equation*}
    C_{[(p_0, m_0)]}^{-1} \lambda^{4 + \nu_{p_0} + \nu_{m_0}}e^{-2i \lambda (x_{p_0} - x_{m_0})} \det S_V(-\lambda) = 1 + o(1)
\end{equation*}
and hence one calculates that 
\begin{equation*}
    \ell_V(t) = 2(x_{p_0} - x_{m_0})t - (4 + \nu_{p_0} + \nu_{m_0}), \quad t_0 < t < t_1.
\end{equation*}
This proves that the first discontinuity of $\ell_V'$ is indeed $t_0$, and has the form given in Theorem \ref{thm:1.8}. Moreover, we calculate $\ell_V'(t_0+) - \ell_V'(t_0-) = 2(x_{p_0} - x_{m_0})$. 

We plan to mimic the foregoing, using (\ref{eq:414}) in place of (\ref{eq:413}). We claim that the next discontinuity occurs at $t_1$ defined in (\ref{eq:415}). To see this, we let $S_1$ denote all the pairs $(p, m) \in S/S_0$ with 
\begin{equation*}
    \frac{\nu_p + \nu_m - (\nu_{p_0} + \nu_{m_0})}{2[(x_p - x_m) - (x_{p_0} - x_{m_0})]} = t_1,
\end{equation*}
and take $(p_1, m_1) \in S_1$ to be the pair such that $x_{o_1} - x_{m_1}$ is maximized. Using that for any other pair $(p, m) \in S_1$, we have 
\begin{equation*}
    \frac{\nu_{p_1} + \nu_{m_1} - (\nu_p + \nu_m)}{2[(x_{p_1} - x_{m_1}) - (x_p - x_m)]} = t_1
\end{equation*}
we see that as $r \to \infty$ with $\lambda = re^{i \theta_t(r)}$ and $t > t_1$, 

\begin{equation*}
    \abs{\lambda^{\nu_{p_1} + \nu_{m_1} - (\nu_{p_0} + \nu_{m_0})} e^{-2i \lambda [(x_{p_1} - x_{m_1}) - (x_{p_0} - x_{m_0})]}} = o(1)
\end{equation*}
and 
\begin{equation*}
    \abs{\lambda^{\nu_{p_1} + \nu_{m_1} - (\nu_{p} + \nu_{m})} e^{-2i \lambda [(x_{p_1} - x_{m_1}) - (x_{p} - x_{m})]}} = o(1).
\end{equation*}
Multiply \eqref{eq:414} by $C_{[(p_0, m_0)]}C^{-1}_{[(p_1, m_1)]}\lambda^{\nu_{p_1} + \nu_{m_1} - (\nu_{p_0} + \nu_{m_0})}e^{-2i \lambda[(x_{p_1} - x_{m_1}) - (x_{p_0} - x_{m_0})]}$ to find that as $r \to \infty$ with $\lambda = re^{i \theta_t(r)}$ and $t > t_1$, 

\begin{align*}
    C_{[(p_1, m_1)]}^{-1}& \lambda^{4 + \nu_{p_1} + \nu_{m_1}}e^{-2i\lambda(x_{p_1} - x_{m_1})}\det S_V(-\lambda)\\
    =&\; 1 + \sum_{(p, m) \in S/(S_0 \cup S_1)} \frac{C_{[(p_1, m_1)]}^{-1} C_{[(p, m)]}}{\lambda^{\nu_p + \nu_m - (\nu_{p_1} + \nu_{m_1})}}e^{2i \lambda[(x_p - x_m) - (x_{p_1} - x_{m_1})]}(1 + o(1)) + o(1) \tag{3.16} \label{eq:416}
\end{align*}

For $(p, m) \in S/(S_0 \cup S_1)$, we have 
\begin{equation*}
    t_1 < \frac{\nu_p + \nu_m - (\nu_{p_1} + \nu_{m_1})}{2[(x_p - x_m) - (x_{p_1} - x_{m_1})]}.
\end{equation*}
If we define 
\begin{equation*}
    t_2 := \min_{(p, m) \in S/(S_0 \cup S_1)} \frac{\nu_p + \nu_m - (\nu_{p_1} + \nu_{m_1})}{2[(x_p - x_m) - (x_{p_1} - x_{m_1})]},
\end{equation*}
then it follows from (\ref{eq:416}) that as $r \to \infty$ with $\lambda = re^{i \theta_t(r)}$ and $t_1 < t < t_2$, we have 
\begin{equation*}
    C_{[(p_1, m_1)]}^{-1} \lambda^{4 + \nu_{p_1} + \nu_{m_1}}e^{-2i \lambda (x_{p_1} - x_{m_1})} \det S_V(-\lambda) = 1 + o(1)
\end{equation*}
and hence 
\begin{equation*}
    \ell_V(t) = 2(x_{p_1} - x_{m_1})t - (4 + \nu_{p_1} + \nu_{m_1}), \quad t_1 < t < t_2.
\end{equation*}
This proves that $t_1$ is indeed the next discontinuity of $\ell_V'$, and has the form claimed in Theorem \ref{thm:1.8}. Moreover, $\ell_V'(t_1 +) - \ell_V'(t_1-) = 2[(x_{p_1} - x_{m_1}) - (x_{p_0} - x_{m_0})]$.

Continuing in this way, we find that all discontinuities have the form claimed in Theorem \ref{thm:1.8}. To see the upper bound of $N(N+1)/2$ on the number of discontinuities, we only need to note that this is the maximal number of terms that can appear in the sum in the expansion (\ref{eq:413}) (the case in which no terms absorb, and all lengths of subintervals are unique). Since each step in the above process reduces the number of terms in the subsequent expansion, it is obvious that we can have at most $N(N+1)/2$ discontinuities. This completes the proof of Theorem \ref{thm:1.8}.

\section{A POTENTIAL PRODUCING INFINITELY MANY STRINGS OF RESONANCES}\label{sec:4}

Here we use methods from the prior section to construct a potential $V \in L_c^\infty(\R; \C)$ such that $-\frac{d^2}{dx^2} + V$ has infinitely many distinct strings of resonances. It is easy to see from the construction that it can also be used to produce examples of potentials with any finite number of strings of resonances. 

To begin the construction, we make the following assumption on $V$.

\begin{assumption}\label{assumption:5.1}
    Let $b_0 = 0 < b_1 < \cdots < b_k < b_{k+1} \nearrow 1$, let $0 < \beta_1 < \beta_2 < \cdots < \beta_k < \beta_{k+1} < \cdots$ be integers, and let $\set{B_k}_{k=0}^\infty \subset \C /\set{0}$ with $\abs{B_k} \leq 1$ for all $k$. Then we assume that $V \in L_c^\infty(\R; \C)$ has the form 
    \begin{equation*}
        V(x) := \begin{cases}
            B_k(x - b_k)^{\beta_{k+1}}(b_{k+1} - x)^{\beta_{k+1}}, & x \in [b_k, b_{k+1}]\\
            0, & x \in \R/(0, 1).
        \end{cases}
    \end{equation*}
\end{assumption}

Our technical lemma below holds for any $V$ as in Assumption \ref{assumption:5.1}, but for the construction of a potential producing infinitely many strings of resonances in Theorem \ref{thm:5.3}, we will fix $\set{b_k}_{k=0}^\infty$, require $\set{\beta_k}_{k=1}^\infty$ to satisfy a certain growth condition (see \eqref{eq:508}), and then the $B_k$ will be chosen depending upon these parameters to ensure the convergence of certain sums.

The following lemma, similar to Lemma \ref{lem:4.2}, gives the expansions of $f_\pm^V$ for $V$ as in Assumption \ref{assumption:5.1}, that we will use to obtain an expansion of $\det S_V$ analogously to (\ref{eq:407}), but having an infinite number of terms. 

\begin{lm}\label{lem:5.2}
    Let $V$ satisfy Assumption \ref{assumption:5.1}. Then $f_+^V \in C^{\beta_1 + 1}([0, b_1]) \cap C^{\beta_1 + 1}([b_1, 1])$, and on each of these intervals separately, we have $\partial_x^d f_+^V(x, \lambda) = O(\abs{\lambda}^{-1})$ for $0 \leq d \leq \beta_1 + 1$. Moreover, for each $k = 0, 1, 2, \ldots, f_-^V \in C^{\beta_{k+1} + 1}([b_k, b_{k+1}])$ and for $x \in [b_k, b_{k+1}]$, we have 

     \begin{enumerate}[label=\roman*)] 
        \item $\partial_x^d f_-^V(x, \lambda) = O(\abs{\lambda}^{-1})$ for $0 \leq d \leq \beta_1 + 1$,
        
        \item For $m = 1, 2, \ldots , k$, 
        \begin{equation*}
            \partial_x^d f_-^V(x, \lambda) = \sum_{p=1}^m B_p e^{-2i \lambda(x - b_p)}O_{p, k}(\abs{\lambda}^{d - \beta_p - 2}) + O(\abs{\lambda}^{-1})
        \end{equation*}
        for $\beta_m + 2 \leq d \leq \beta_{m+1} + 1$. Here the $O_{p, k}(\cdot)$ bounds may be taken with constants depending on $p, \beta_{p+1}, b_p, b_{p+1}, k, \beta_{k+1}, b_k, b_{k+1}$, but not on $B_p$ or $B_k$. 
    \end{enumerate}
\end{lm}

\begin{proof}
    The statements about regularity follow from elliptic regularity theory and the definition of $f_\pm^V$ in terms of the resolvent. 

    The statements for $f_-^V$ follow by an induction argument in exactly the same was as in Lemma \ref{lem:4.2}, and are in fact simpler due to our requirement that the $\beta_k$ are increasing; that proof also gives the claim about the $O_{p, k}$ bounds, using in addition that $\abs{B_k} \leq 1$ for all $k$.

    For the statements about $f_+^V$, we use that $V \in C^\infty([0, b_1])$, vanishing at both endpoints up to order $\beta_1$, and that $V \in C^{\beta_2-1}([b_1, 1]) \subset C^{\beta_1}([b_1, 1])$ vanishing up to and including order $\beta_1$, and then apply a proof directly analogous to the first part of the proof of Lemma \ref{lem:4.2} on each of these intervals. This concludes the proof of the lemma. 
\end{proof}

Next, we use integration by parts to compute that, formally, 
\begin{align*}
    \widehat V(-2\lambda) = \sum_{k=0}^\infty \frac{B_k(-1)^{\beta_{k+1}}}{(2i \lambda)^{\beta_{k+1} + 1}} \bigg\{(-1)^{\beta_{k+1}} \beta_{k+1}!&\; (b_{k+1} - b_k)^{\beta_{k+1}} e^{2i\lambda b_{k+1}} - \beta_{k+1}!(b_{k+1} - b_k)^{\beta_{k+1}} e^{2i \lambda b_k}\\
   &\; - \int_{b_k}^{b_{k+1}}e^{2i\lambda x} \partial_x^{\beta_{k+1}+1} \left[(x-b_k)^{\beta_{k+1}}(b_{k+1} - x)^{\beta_{k+1}} \right]dx \bigg\}.
\end{align*}
Hence, for $\operatorname{Im} \lambda \leq 0$, we get 
\begin{equation*}
    \widehat V(-2\lambda) = \sum_{k=0}^\infty \frac{B_k \beta_{k+1}!(b_{k+1} - b_k)^{\beta_{k+1}}}{(2i \lambda)^{\beta_{k+1} + 1}} e^{2i\lambda b_{k+1}}(1 + O_k(\abs{\lambda}^{-1})) + O(\abs{\lambda}^{-\beta_1 - 1}) \tag{4.1} \label{eq:501}
\end{equation*}
where we note that the $O_k(\abs{\lambda}^{-1})$ decaying terms depend upon $k, \beta_{k+1}, b_k,$ and $b_{k+1}$, but not on $B_k$; thus, assuming the $\abs{B_k}$ are sufficiently small (depending upon $k, \beta_{k+1}, b_k, b_{k+1}$), the sum will converge. We also have 
\begin{equation*}
    \widehat V(2\lambda) = \frac{B_0 \beta_1! b_1^{\beta_1}}{(2i\lambda)^{\beta_1 + 1}}(1 + O(\abs{\lambda}^{-1})) + e^{-2i \lambda b_1}O(\abs{\lambda}^{-\beta_1 - 1}). \tag{4.2} \label{eq:502}
\end{equation*}

We now want to obtain similar expansions for the reflection coefficients 
\begin{equation*}
    \rho_\pm^V(-\lambda) = \widehat V(\pm 2\lambda) - \int_0^1 e^{\mp 2i\lambda x}V(x) f_\pm^V(x, \lambda)dx,
\end{equation*}
so we need to bound the integral term on the right. Working first with $\rho_-$ we begin by splitting the integral into a sum of integrals over $[b_k, b_{k+1}]$ and integrating by parts an appropriate number of times on each interval to obtain: 
\begin{align*}
    \int_0^1& e^{2i \lambda x}V(x) f_-^V(x, \lambda)dx = \sum_{k=0}^\infty \frac{B_k(-1)^{\beta_{k+1}}}{(2i \lambda)^{\beta_{k+1} + 1}} \bigg\{ (-1)^{\beta_{k+1}} \beta_{k+1}!(b_{k+1} - b_k)^{\beta_{k+1}} f_-^V(b_{k+1}, \lambda)e^{2i \lambda b_{k+1}}\\
    -&\; \beta_{k+1}!(b_{k+1} - b_k)^{\beta_{k+1}} f_-^V(b_k, \lambda)e^{2i \lambda b_k} - \int_{b_k}^{b_{k+1}} e^{2i\lambda x} \partial_x^{\beta_{k+1} + 1} \left[(x - b_k)^{\beta_{k+1}}(b_{k+1} - x)^{\beta_{k+1}} f_-^V(x, \lambda) \right] dx \bigg \} \tag{4.3} \label{eq:503}
\end{align*}
Now, we may write $\partial_x^{\beta_1 + 1}\left[(x - b_0)^{\beta_1}(b_1 - x)^{\beta_1}f_-^V(x, \lambda)\right] = O(\abs{\lambda}^{-1})$ and for $k \geq 1$, 
\begin{equation*}
    \partial_x^{\beta_{k+1} + 1}\left[(x - b_k)^{\beta_{k+1}}(b_{k+1} - x)^{\beta_{k+1}} f_-^V(x, \lambda) \right] = \sum_{d = 0}^{\beta_1 + 1} g_{k, d}(x) \partial_x^d f_-^V(x, \lambda) + \sum_{m=1}^k \sum_{d = \beta_m + 2}^{\beta_{m+1} + 1} g_{k, d}(x) \partial_x^d f_-^V(x, \lambda)
\end{equation*}
for functions $g_{k, d}(x) = O_k(1)$ with bounds depending only on $k, \beta_{k+1}, b_k, b_{k+1}$. Thus, using Lemma \ref{lem:5.2} and rearranging a sum, if $k \geq 1$, 
\begin{align*}
    \partial_x^{\beta_{k+1} + 1} \left[(x - b_k)^{\beta_{k+1}}(b_{k+1} - x)^{\beta_{k+1}} f_-^V(x, \lambda) \right] =&\; O_k(\abs{\lambda}^{-1}) + \sum_{m=1}^k \sum_{d = \beta_m + 2}^{\beta_{m+1} + 1} g_{k, d}(x) \sum_{p=1}^m B_p e^{-2i \lambda(x - b_p)}O_{p, k}(\abs{\lambda}^{d - \beta_p - 2})\\
    =&\; O_k(\abs{\lambda}^{-1}) + \sum_{p=1}^k B_p e^{-2i \lambda (x - b_p)}O_{p, k}(\abs{\lambda}^{\beta_{k+1} - \beta_p - 1})
\end{align*}
where the $O_k(\cdot)$ is bounded in terms of $k, \beta_{k+1}, b_k, b_{k+1}$, and the $O_{p, k}(\cdot)$ in terms of $p, \beta_{p+1}, b_p, b_{p+1}, k, \beta_{k+1}, b_k, b_{k+1}$. Using this in (\ref{eq:503}) and rearranging another sum, we easily obtain 
\begin{equation*}\tag{4.4} \label{eq:504}
\begin{aligned}
    \int_0^1 e^{2i \lambda x} V(x) f_-^V(x, \lambda) dx =&\; e^{2i \lambda b_1} O(\abs{\lambda}^{-\beta_1 - 2}) + \sum_{k=1}^\infty \bigg\{ B_k e^{2i \lambda b_{k+1}} O_k(\abs{\lambda}^{-\beta_{k+1} - 2}) \\
    +&\; B_k \sum_{p=1}^k B_p e^{2i \lambda b_p} O_{p, k}(\abs{\lambda}^{- \beta_p - 2}) \bigg\}\\
    =&\; \sum_{k=0}^\infty B_k e^{2i \lambda b_{k+1}} O_k( \abs{\lambda}^{- \beta_{k+1} - 2}) + \sum_{p=1}^\infty B_p e^{2i \lambda b_p} \sum_{k=p}^\infty B_k O_{p, k}(\abs{\lambda}^{- \beta_p - 2}) 
    \end{aligned}
\end{equation*}

Taking the $\abs{B_k}$ yet smaller if necessary, all sums in \eqref{eq:504} converge. Thus, from \eqref{eq:504} and \eqref{eq:501}, we have

\begin{equation*}
    \rho_-^V(-\lambda) = \sum_{k=0}^\infty \frac{B_k \beta_{k+1}!(b_{k+1} - b_k)^{\beta_{k+1}}}{(2i \lambda)^{\beta_{k+1} + 1}}e^{2i \lambda b_{k+1}}(1 + O_k( \abs{\lambda}^{-1})) + O(\abs{\lambda}^{-\beta_1 - 1}) \tag{4.5} \label{eq:505}
\end{equation*}
with $O_k(\cdot)$ bounded in terms of $k, \beta_{k+1}, b_k, b_{k+1}$. Analogously, using the bounds for $f_+^V$ from Lemma \ref{lem:5.2} and (\ref{eq:502}), we get 
\begin{equation*}
    \rho_+^V(-\lambda) = \frac{B_0 \beta_1! b_1^{\beta_1}}{(2i \lambda)^{\beta_1 + 1}}(1 + O(\abs{\lambda}^{-1})) + e^{-2i \lambda b_1}O(\abs{\lambda}^{- \beta_1 - 1}). \tag{4.6} \label{eq:506}
\end{equation*}

From (\ref{eq:405}), (\ref{eq:505}), (\ref{eq:506}), we obtain the following expansion of the scattering determinant in $\operatorname{Im} \lambda \leq 0$:
\begin{equation*}
    \det S_V(-\lambda) = 1 + \sum_{k=1}^\infty \frac{C_k}{(2i \lambda)^{4 + \beta_1 + \beta_k}} e^{2i \lambda b_k}(1 + O_k(\abs{\lambda}^{-1})) + O(\abs{\lambda}^{-1}), \tag{4.7} \label{eq:507}
\end{equation*}
where $C_k := -B_0\beta_1!b_1^{\beta_1} B_{k-1}\beta_k!(b_k - b_{k-1})^{\beta_k}$, and the $O_k(\cdot)$ depend only on $k, \beta_{k+1}, b_k, b_{k+1}$. 

We now make the following additional requirements on the $\beta_k$.

\begin{equation*}\tag{4.8} \label{eq:508}
\begin{aligned}
    i)\;\;\frac{4 + 2 \beta_1}{2b_1} < \frac{\beta_2 - \beta_1}{2(b_2 - b_1)} < \frac{\beta_3 - \beta_2}{2(b_3 - b_2)} < \cdots < \frac{\beta_{k+1} - \beta_k}{2(b_{k+1} - b_k)} <\frac{\beta_{k+2} - \beta_{k+1}}{2(b_{k+2} - b_{k+1})}  < \cdots\\
    ii)\;\; \frac{4 + 2 \beta_1}{2b_1} \leq \frac{4 + \beta_1 + \beta_k}{2b_k}
        \text{ for all } k \geq 1, \text{ and for } p = 1, 2, 3, \ldots , \frac{\beta_{p+1} - \beta_p}{2(b_{p+1} - b_p)} \leq\\ \frac{\beta_{k+1} - \beta_p}{2(b_{k+1} - b_p)} \text{ for all } k \geq p. 
\end{aligned}    
\end{equation*}

We note that once the $\set{b_k}$ are fixed, these requirements on $\set{\beta_k}$ can be made inductively, simply requiring each $\beta_k$ to be large enough depending upon the previous ones. 

Defining $t_1 := \frac{4 + 2\beta_1}{2b_1}, t_k := \frac{\beta_k - \beta_{k-1}}{2(b_k - b_{k-1})}$ for $k \geq 2$, we may now prove:

\begin{thm}\label{thm:5.3}
    Let $V$ satisfy Assumption \ref{assumption:5.1}, and assume, in addition, that the $\beta_k$ satisfy the requirements in (\ref{eq:508}), and $\abs{B_k}$ are chosen sufficiently small. Then for each $k = 1, 2, 3, \ldots$, there is a sequence of resonances asymptotic to $\set{\operatorname{Im} \lambda = -t_k \log(1 + \abs{\lambda}), \operatorname{Re} \lambda > 0}$, with linear density 
    \begin{equation*}
        \frac{1}{\pi}(b_k - b_{k-1}).
    \end{equation*}
\end{thm}

\begin{proof}
    We will use (\ref{eq:507}) to show that for $t < t_1$, the limit 
    \begin{equation*}
        \ell_V(t) = \lim_{r \to \infty} \frac{\log \abs{\det S_V(-re^{i \theta_t(e)})}}{ \log r} = 0, \tag{4.9} \label{eq:509}
    \end{equation*}
    exists, and for each $k \geq 1$, the limit 
    \begin{equation*}
        \ell_V(t) = \lim_{r \to \infty} \frac{\log \abs{\det S_V(-re^{i \theta_t(r)})}}{\log r} = 2b_kt - (4 + \beta_1 + \beta_k) \tag{4.10} \label{eq:510}
    \end{equation*}
    exists for all $t_k < t < t_{k+1}$. The conclusion then follows by application of Theorem \ref{thm:1.3}. 

    To prove (\ref{eq:509}), we note that if $t < t_1$, then as $r \to \infty$ with $\lambda = re^{i \theta_t(r)}$, we have 
    \begin{equation*}
        \abs{\lambda^{-4 - \beta_1 - \beta_k} e^{2i \lambda b_k}} = o(1)
    \end{equation*}
    uniformly in $k$. Hence from (\ref{eq:507}), 
    \begin{equation*}
        \det S_V(- re^{i \theta_t(r)}) = 1 + o(1)
    \end{equation*}
    as $r \to \infty$, and (\ref{eq:509}) follows. 

    Next, we claim that for $p = 1,2, 3, \ldots$, as $r \to \infty$ with $\lambda = re^{i \theta_t(r)}$ and $t > t_p$, we have 
    \begin{equation*}
        \det S_V(-\lambda) = \frac{C_p}{(2i \lambda)^{4 + \beta_1 + \beta_p}}e^{2i \lambda b_p} \left(1 + \sum_{k = p+1}^\infty \frac{C_k/C_p}{(2i \lambda)^{\beta_k - \beta_p}}e^{2i \lambda(b_k - b_p)} (1 + O_k(\abs{\lambda}^{-1})) + o(1)\right). \tag{4.11} \label{eq:511}
    \end{equation*}
    To see this, we first prove the case $p = 1$. Note that if $t > t_1$, then as $r \to \infty$ with $\lambda = re^{i \theta_t(r)}$, we have $\abs{\lambda^{4 + 2\beta_1}e^{-2i \lambda b_1}} = o(1)$, so pulling this term out in (\ref{eq:507}) we easily obtain the case $p =1 $ in (\ref{eq:511}). Assuming (\ref{eq:511}) proven for some $p$, we note that $\abs{\lambda^{\beta_{p+1} - \beta_p}e^{-2i \lambda (b_{p+1} - b_p)}} = o(1)$ if $t > t_{p+1}$, so we may pull this term of out (\ref{eq:511}) to prove the case $p + 1$. Thus (\ref{eq:511}) holds for all $p$. 

    Fixing $p$, we note that if $t_p < t < t_{p+1}$, then as $r \to \infty$ with $\lambda = re^{i \theta_t(r)}$, we have $\abs{\lambda^{-(\beta_k - \beta_p)} e^{2i \lambda(b_k - b_p)}} = o(1)$ uniformly in $k \geq p+1$. From (\ref{eq:511}), then, 
    \begin{equation*}
        \det S_V(-\lambda) = \frac{C_p}{(2i\lambda)^{4 + \beta_1 + \beta_p}}e^{2i \lambda b_p}(1 + o(1))
    \end{equation*}
    and (\ref{eq:510}) follows from this. The proof is complete. 
\end{proof}

\bibliographystyle{amsplain}

\bibliography{references2}

\noindent \emph{Email address}: travisdcunningham@gmail.com

\end{document}